\documentclass{article}

\usepackage{moreverb}

\usepackage{multirow,graphicx,capt-of}
\usepackage{amsmath,amsfonts,amssymb,amsthm}

\makeatletter
\newif\if@restonecol
\makeatother

\usepackage[algo2e,ruled,lined,titlenumbered]{algorithm2e}

\newtheorem{prop}{Proposition}
\newtheorem{cor}{Corollary}

\DeclareMathAlphabet{\mathdj}{U}{msb}{m}{n}
\DeclareMathAlphabet{\mathbf}{OML}{cmm}{b}{it}
\DeclareMathAlphabet{\mathbfsl}{OT1}{cmr}{bx}{sl}
\newcommand{\A}{\ensuremath{\mathdj{A}}}
\newcommand{\R}{\ensuremath{\mathdj{R}}}

\newcommand{\mK}{\ensuremath{\mathcal{K}}}
\newcommand{\multlag}{ \ensuremath{\lambda}}

\newcommand{\Ks}{\ensuremath{\mathcal{K}}}

\newcommand{\Cs}{\ensuremath{\mathcal{C}}}

\newcommand{\iden}{\ensuremath{\mathbf{I}}}
\newcommand{\Id}{\iden}
\newcommand{\op}{\ensuremath{\mathbf{A}}}
\newcommand{\opk}{\ensuremath{\mathbf{A}^{(k)}}}

\newcommand{\precok}{\ensuremath{{\mathbf{M}^{(k)}}}}
\newcommand{\precoM}{\ensuremath{\mathbf{M}}}
\newcommand{\preco}{\ensuremath{\mathbf{M}^{-1}}}
\newcommand{\const}{\ensuremath{\mathbf{C}}}
\newcommand{\constD}{\ensuremath{\mathbf{D}}}

\newcommand{\constE}{\ensuremath{\mathbf{E}}}
\newcommand{\constk}{\ensuremath{{\mathbf{C}^{(k)}}}}
\newcommand{\proj}{\ensuremath{\mathbf{P}}}

\newcommand{\range}{\ensuremath{\operatorname{Range}}}

\newcommand{\Wdes}{\ensuremath{\mathbf{W}}}
\newcommand{\Zp}{\ensuremath{\mathbf{Z}}}

\newcommand{\Vritz}{\ensuremath{\mathbf{V}}}
\newcommand{\Yritz}{\ensuremath{\mathbf{Y}}}
\newcommand{\Hess}{\ensuremath{\mathbf{H}}}
\newcommand{\bTheta}{\ensuremath{\boldsymbol{\Theta}}}
\newcommand{\precoL}{\ensuremath{\mathbf{L}}}
\newcommand{\bQ}{\ensuremath{\mathbf{Q}}}

\title{Total and selective reuse of Krylov subspaces for the resolution of sequences of nonlinear structural problems}

\author{P. Gosselet, C. Rey, J. Pebrel\\ 
LMT Cachan, ENS Cachan/CNRS/UPMC/PRES UniverSud Paris\\  61 Avenue du Pr\'esident Wilson, 94235 Cachan, France}

\begin{document}

\maketitle




\begin{abstract}
This paper deals with the definition and optimization of augmentation spaces for faster convergence of the conjugate gradient method in the resolution of sequences of linear systems. Using advanced convergence results from the literature, we present a procedure based on a selection of relevant approximations of the eigenspaces for extracting, selecting and reusing information from the Krylov subspaces generated by previous solutions in order to accelerate the current iteration. Assessments of the method are proposed in the cases of both linear and nonlinear structural problems.

Keywords: Krylov solvers; multiresolution; model reduction.
\end{abstract}


\section{Introduction}
Accelerating the convergence of Krylov iterative solvers \cite{SAAD:2000:IMS} is an old issue which has returned to the spotlight because of the increasing number of applications for which these are preferred to direct solvers today. Traditional approaches aim at improving the condition number by using frameworks in which efficient preconditioners exist (e.g. domain decomposition methods \cite{KLAWONN:2001:FNN,GOSSELET.2007.1}), or for which good initialization vectors \cite{GOSSELET:2003:IEI}, relevant augmentation subspaces \cite{DOSTAL:1988:CGM,SAAD:1997:AAK,CHAPMAN:1996:DAK} or suitable block strategies (see \cite{Aliaga00alanczos-type} for a very general block-Lanczos algorithm) are available. For instance:
\begin{itemize}
  \item For 3D elasticity problems, domain decomposition methods come with ``physical'' augmentation associated with the global equilibrium of floating substructures (rigid body motions), which makes the methods scalable \cite{FARHAT:1994:ADV,MANDEL:1996:COEF}; for plate and shell problems, additional augmentation through ``corner modes''  \cite{FARHAT:1996:COR1,FARHAT:1996:COR2,LETALLEC:1997:SHELL} is required.
  \item For structures with repeated patterns, block strategies are possible \cite{GOSSELET.2009.1}.
  \item For restarted algorithms, one can use deflation or augmentation \cite{MORGAN:1995:Eigen,ERHEL:1995:RGM}, or block techniques \cite{baker:1608}.
  \item For problems with multiple right-hand sides, deflation \cite{FARHAT.1994.ESBI,SAAD:2000:DCG,ERHEL:2000:ACG} is a rather classical approach.
\end{itemize}
The problem with these techniques is that they require some \emph{a priori} information which is seldom available, except in specific cases.

Many recent works present theoretical and practical comparisons of the numerous algorithms which have been developed in connection with these ideas \cite{TANG:2007:TNC,SIMONCINI:2007:RCD}.

Multiresolution approaches form a general framework in which numerical information is available to accelerate the convergence of Krylov solvers. Multiresolution refers to situations in which the solution of a mechanical problem cannot be achieved through the resolution of a single linear system. For example, calculating the solution of a nonlinear or time-dependent problem or exploring a design of experiment during an optimization procedure requires the resolution of sequences of linear systems. Multiresolution is more general than using multiple right-hand sides because the matrices themselves are likely to change from one system to another (multiple right-hand and left-hand sides). Thus, the problem consists in solving a $k$-indexed family of large, sparse, linear $n\times n$ systems of the form:
\begin{equation}
  \op^{(k)} x^{(k)} =  b^{(k)}
\end{equation}
Although different, the systems are assumed to be similar to one another. This similarity can be defined in several ways: in terms of rank, by the fact that $\operatorname{rank}(\op^{(k)}-\op^{(k-1)})\ll \operatorname{rank}(\op^{(k)})$; or in a spectral sense by the fact that the eigenspaces remain stable from one system to another; or in terms of the Krylov subspaces generated \cite{CARPRAUX:1994:SKB}. The first case can be dealt with easily, even with direct solvers, by using the Sherman-Morrison formula; the second case requires augmentation strategies in order to eliminate the most penalizing part of the spectrum and improve the active condition number\footnote{``active'' referring to the part of the spectrum of the matrix which is solicited by the right-hand side.} \cite{Rey:1998:RRP,GOSSELET:2002:SRK,PARKS:2006:RKS,WANG:2007:LST}; and the last case calls for preconditioning techniques \cite{REY:1996:TAR,Risler:2000:IAA}.

While most of the studies of Krylov methods for multiresolution (often referred to as the recycling of Krylov subspaces) are set in the framework of GMRes/MinRes \cite{PARKS:2006:RKS,WANG:2007:LST}, we chose to work on the specific case of the resolution of symmetric, positive definite systems using conjugate gradients (CGs), in which the convergence is under control and related to easily calculated spectral properties \cite{VANDERSLUIS:1986:RCCG}.
In earlier works, the authors developed efficient preconditioners based on previous Krylov subspaces \cite{REY:1996:TAR,Risler:2000:IAA} which took advantage of the conjugation properties of CGs,
but did not extract the most interesting part of the information available in the Krylov subspaces, and they proposed augmentation techniques using Ritz vectors \cite{Rey:1998:RRP}.
Typically, these works were aimed at nonlinear mechanical systems solved by Newton-Raphson linearization and FETI or BDD domain decomposition \cite{GOSSELET:2002:DDM,lene:2001}. 

The recycling of Krylov subspaces can also be analyzed from the model reduction point of view. Since Krylov solvers satisfy Petrov-Galerkin conditions, they share many common points with strategies based on Karhunen-Loeve expansion \cite{meyer.2003.1,Ryckelynck.2006.1,Nouy.2007.1}. These similarities are well-known \cite{Freund2000395,Freund.2003.1,heres.2007.1}. But our objective is not to develop reduced models of mechanical systems in order to perform fast but coarse analyzes; it is to define, improve and reuse reduced models in order to carry out calculations both rapidly and accurately.

In this paper, we undertake a more in-depth investigation of augmentation using a selection of post-processed Ritz vectors. In Section 2, we begin with a detailed presentation of the theoretical framework of the augmented preconditioned conjugate gradient method; then, in Section 3, we propose a first reuse algorithm in a multiresolution framework (TRKS); in Section 4, we improve this algorithm by proposing a procedure for selecting the ``best'' Ritz vectors (SRKS and ``cluster''); finally, in Section 5, we propose an evaluation of the method in the case of nonlinear mechanics and parametric problems, using domain decomposition methods \cite{GOSSELET.2007.1} to define efficient preconditioners.

\section{The augmented preconditioned conjugate gradient method}\label{section:APCG}

\subsection{Algorithm and properties}
Let us consider the linear problem
\begin{equation}\label{pb1}
\op x = b,
\end{equation}
where $\op$ is an $n \times n$ symmetric positive definite matrix, and let us study the resolution of this system using the augmented preconditioned conjugate gradient algorithm. With $\precoM$ being the $n \times n$ symmetric positive definite matrix of the preconditioner, we introduce the following notations:
\begin{equation}
\begin{aligned}
i=0\ldots m &\qquad & \text{the iteration number}\\
x_i & & \text{the } i^{th} \text{ approximation}\\
r_i=b-Ax_i = A(x-x_i) & & \text{the } i^{th} \text{ residual}
\end{aligned}
\end{equation}
With no loss of generality, the presentation can be limited to the case of a zero initial guess $x_{00}=0$. (Otherwise, one can set $b \leftarrow b-\op x_{00}$.)

Let $\Cs$ be a subspace of $\mathbb{R}^n$ of dimension $n_c$, and let Matrix $\const=\left[c_1, \ldots, c_{n_c}\right]$ be a basis of $\Cs$. The search principle of the augmented left-preconditioned conjugate gradient is:
\begin{equation}\label{eq:krysol_principle}
\left\{\begin{array}{ll}
\textrm{find } &x_i \in \Ks_i(\precoM^{-1}\op,\Cs,\precoM^{-1}r_0)\\
\textrm{such that }&r_i \perp \Ks_i(\precoM^{-1}\op,\Cs,\precoM^{-1}r_0)
\end{array}\right.
\end{equation}
where $\Ks_i(\precoM^{-1}\op,\Cs,\precoM^{-1}r_0)$ is the augmented Krylov subspace associated with preconditioned operator $\precoM^{-1}\op$ and augmentation subspace $\Cs$:
\begin{equation}
\Ks_i(\precoM^{-1}\op,\Cs,\precoM^{-1}r_0) = \\
\operatorname{span}\left(\precoM^{-1}r_0,\ldots,(\precoM^{-1}\op)^{(i-1)}\precoM^{-1}r_0\right)\oplus\Cs
\end{equation}

A classical implementation relies on the definition of a convenient initialization and projector pair $(x_0$,$\proj)$:
\begin{equation}
x=x_0 + \proj y \quad
\left\{ \begin{array}{lll}
\const^Tr_0=0 &\Leftarrow &x_0=\const(\const^T\op \const)^{-1}\const^Tb \\
\const^T\op \proj=0 &\Leftarrow& \proj=\iden-\const(\const^T\op \const)^{-1}\const^T\op
\end{array}\right.
\end{equation}
One should note that since $\op \proj=\proj^T\op \proj=\proj^T\op$ augmentation preserves symmetry. One should also note that $\proj\const=0$. The system to be solved is:
\begin{equation}\label{sys1}
\op \proj y= (\proj^{T} \op \proj) y=r_0=\proj^{T}b
\end{equation}

The $\const$-augmented, $\precoM$-preconditioned conjugate gradient technique (APCG) implemented by projection is presented in Algorithm \ref{APCG}. (For the sake of simplicity, the methods will be described assuming exact arithmetic, even though they are compatible with more realistic full reorthogonalization \cite{LINGEN:2000:EGS}.)

\begin{algorithm2e}[th]\caption{APCG($\op, \precoM,\const$, $b$)}\label{APCG}
Calculate $\op \const$, $(\const^T \op \const)^{-1}$ \quad; \quad($\proj=\Id - \const\left(\const^T \op \const\right)^{-1}\const^T \op$)\;  %
$x_0 = \const(\const^T \op \const)^{-1}\const^T b$\;%
$r_0=b-\op x_0=\proj^T b$\;%
$z_0 = \proj \preco r_0$, $w_0=z_0$\;%
\For{$j=1,\ldots,m$}{%
  $\alpha_{j-1}=(r_{j-1},w_{j-1})/(\op w_{j-1},w_{j-1})$\\
  $x_{j}=x_{j-1}+\alpha_{j-1} w_{j-1}$\\
  $r_{j}=r_{j-1}-\alpha_{j-1} \op w_{j-1}$\\
  $z_{j} = \proj \preco r_{j+1}$\\
  $w_{j}=z_{j} - \beta_j w_{j-1}$\\
  $\beta_j=(\op w_{j-1},z_{j})/(w_{j-1},\op w_{j-1})$%
}%
\end{algorithm2e}

The following basic relations hold:
\begin{equation}\label{relation:acg}
\begin{aligned}
(r_i,z_j)&=0, \quad i \neq j\\
(w_i,\op w_j)&=0, \quad i \neq j
\end{aligned}
\end{equation}
With $\Wdes_i=\left[w_0,\ldots,w_{i-1}\right]$ and $\Zp_i=\left[z_0,\ldots,z_{i-1}\right]$ being two bases of $\Ks_i(\proj\precoM^{-1}\op,z_0)$, the projector enables the spaces to be divided orthogonally:
\begin{equation}\Ks_i(\precoM^{-1}\op,\Cs,\precoM^{-1}r_0) = \Ks_i(\proj\precoM^{-1}\op,z_0) \overset{\perp_A}{\oplus} \Cs\end{equation}

Of course, in the absence of optional constraints ($\const=0$, $\proj=\Id$), APCG reduces to standard preconditioned conjugate gradients PCG($\op$,$\precoM$, $b$); if, in addition, $\preco=\Id$, it becomes a standard conjugate gradient algorithm CG($\op$, $b$).

Let us recall a first result which was proven in \cite{DOSTAL:1988:CGM} for the case of non-preconditioned augmented conjugate gradients.
\begin{prop}\label{prop:acg} Let $\mathcal{V}=\range(\proj)$
\item[$\bullet$] APCG($\op, \Id,\const$, $b$) is equivalent to CG(${{\proj}}^{T} {\op} {\proj}_{|\op\mathcal{V}}$, $\proj^T b$) in the sense that both generate the same residuals. $x_i$, the $i^{th}$ APCG approximation, is connected to $y_i$, the $i^{th}$ CG approximation, by $x_i=x_0+\proj y_i$.
\item[$\bullet$] APCG($\op, \Id,\const$, $b$) does not break down; it converges, and its asymptotic convergence rate is governed by the condition number $\kappa\left( {{\proj}}^{T} {\op} {\proj}_{|{\op} {\mathcal{V}}} \right) \leqslant \kappa\left({\op} \right)$.%
\end{prop}

Consequently, augmentation strategies never decrease the asymptotic convergence rate. The following corollary is straightforward:
\begin{cor}\label{prop:acg_e}
Let $\mathbf{D}=\left[ d_1, \ldots, d_{m_d}\right]$ be a set of $m_d$ linearly independent vectors such that $\constE=\left[\const, \mathbf{D} \right]$ is a full column rank matrix. Let $\proj_{\constE}=\iden-\constE(\constE^T\op \constE)^{-1}\constE^T\op$, and let $\mathcal{V}_{\constE}$ be the range of $\proj_{\constE}$.
\item[$\bullet$] APCG($\op, \Id, \constE$, $b$) is equivalent to APCG($\proj^T \op \proj$, $\Id$, $\mathbf{D}$, $\proj^T b$) in the sense that both generate the same residual. $x^{\constE}_i$, the $i^{th}$ approximation of APCG($\op, \Id, \constE$, $b$), is connected to $x^\constD_i$, the $i^{th}$ approximation of APCG($\proj^T \op \proj$, $\Id$, $\mathbf{D}$, $\proj^T b$), by $x^\constE_i=x_0+\proj x^\constD_i$.
\item[$\bullet$] The asymptotic convergence rate is governed by $\kappa\left( {{\proj}_{\constE}}^{T} {\op} {\proj_{\constE}}_{|{\op} {\mathcal{V}_{\constE}}} \right) \leqslant \kappa\left({{\proj}}^{T} {\op} {\proj}_{|{\op} {\mathcal{V}}} \right)\leqslant \kappa\left({\op}\right)$.
\end{cor}
In conclusion, an increase in the size of the augmentation can only improve the asymptotic rate of convergence. (In the worst case, it leaves it unchanged.)

Now let us focus on the effect of preconditioning. Since $\precoM$ is a symmetric positive definite matrix, it can be factorized in Cholesky form $\precoM=\precoL \precoL^{T}$ (where $\precoL$ denotes a lower triangular matrix with positive diagonal coefficients). Let us introduce the notation:
\begin{equation}
\begin{array}{c}
\hat{\op}=\precoL^{-1}\op \precoL^{-T} \quad; \quad
\hat{b}=\precoL^{-1} b \quad; \quad \hat{x}=\precoL^T x\\
\hat{\const}=\precoL^{T}\const \quad; \quad \hat{\proj}=\Id-\hat{\const}(\hat{\const}^T\hat{\op} \hat{\const})^{-1}\hat{\const}^T\hat{\op}\\
\end{array}
\end{equation}
Then, the following equivalence between preconditioned and non-preconditioned augmented conjugate gradients holds:
\begin{prop}\label{prop:pcg}
APCG($\op, \precoM,\const$, $b$) is equivalent to APCG($\hat{\op}$, $\Id$, $\hat{\const}$, $\hat{b}$) with $\hat{r}=\precoL^{-1}r=\hat{z}=\precoL^{T}z$, $\hat{w}=\precoL^{T}w$, $\hat{\alpha}=\alpha$ and $\hat{\beta}=\beta$. Its asymptotic convergence rate is governed by
$\kappa\left( {\hat{\proj}}^{T} \hat{\op} \hat{\proj}_{|\hat{\op}
\hat{\mathcal{V}}} \right) \leqslant \kappa\left(\hat{\op} \right)$.
\end{prop}%
\begin{proof} Since $\hat{\proj}=\precoL^{T} \proj \precoL^{-T}$, we obtain directly $\hat{x_{0}}=\hat{\const}(\hat{\const}^T \hat{\op} \hat{\const})^{-1}\hat{\const}^T \hat{b}=\precoL^T x_0$, $\hat{r_{0}}=\hat{b}-\hat{\op}\hat{x_{0}}=\precoL^{-1} r_0=\hat{z_{0}}=\precoL^T z_0$ and $\hat{w_{0}}=\precoL^T w_0$. By induction, it follows that $\hat{\alpha}_{j-1}=(\hat{r}_{j-1},\hat{z}_{j-1})/(\hat{\op} \hat{w}_{j-1},\hat{w}_{j-1})={\alpha}_{j-1}$, $\hat{r_{j}}=\precoL^{-1} r_{j}$, $\hat{\beta}_{j}=(\hat{\op}\hat{w}_{j-1},\hat{z}_{j})/(\hat{\op} \hat{w}_{j-1},\hat{w}_{j-1})={\beta}_{j}$ and $\hat{w_{j}}=\precoL^T w_j$.
Proposition \ref{prop:acg} provides the inequality concerning the asymptotic convergence rate.
\end{proof}
Putting these propositions together, APCG($\op, \precoM,\const$, $b$) is equivalent to CG($\precoL^{-1}\proj^T \op \proj \precoL^{-T}$, $\precoL^{-1}\proj^Tb$).
%
All these results lead us to propose an efficient augmentation by analogy with an equivalent, simpler system solved by classical conjugate gradients.

\subsection{Interpretation and choice of the augmentation}

From a ``constraint'' point of view, the projection guarantees the $\const$-orthogonality of the residual throughout the iterations ($\const^T r_j=0$). For example, in the FETI domain decomposition method, the residual is the displacement jump between the subdomains; in the case of shell and plate problems, matrix $\const$ is introduced to enforce the continuity of the displacement at the corner points \cite{FARHAT:2000:FETINL}. In the BDD domain decomposition method, matrix $\const$ is associated with the rigid body motions of floating substructures and, therefore, local Neumann problems in the preconditioner are always well-posed \cite{MANDEL:1993:BAL}; for shell and plate problems, the matrix is enriched by corner mode corrections \cite{LETALLEC:1997:SHELL}. In both cases, matrix $(\const^T\op \const)^{-1}$, called a coarse grid matrix, plays a crucial role in the scalability of these methods.

From a ``spectral'' point of view, augmentation can be used to decrease the active condition number (``active'' referring to eigenelements solicited by the right-hand side) and, thus, improve the asymptotic convergence rate. This is called a deflation strategy \cite{ERHEL:2000:ACG,CHAPMAN:1996:DAK}, which boils down to building matrix $\const$ by using (approximate) eigenvectors associated with the lowest eigenvalues. Obviously, when $\const$ consists of the $n_c$ eigenvectors associated with the lowest eigenvalues $\left(\lambda_1\leqslant\ldots\leqslant\lambda_{n_c}\leqslant\ldots\leqslant\lambda_n\right)$, the condition number decreases strictly: $\kappa\left(\proj^T\op \proj_{\op \mathcal{V}}\right)=\frac{\lambda_n}{\lambda_{n_c}}<\frac{\lambda_n}{\lambda_{1}}$.

From a ``model reduction'' point of view, subspace $\Cs$ represents a ``macro'' (or coarse) space in which the macro part of the solution is calculated directly during the initialization while the ``micro'' part of the solution, when required, is obtained during the iterations.

\subsection{Estimation of computation costs}
With regard to the numerical cost of augmentation, the main operations for the construction of the projector are: (i) the block product $\op \const$ (and assembly with neighbors for domain decomposition methods), (ii) the block dot-product $(\const^T \op \const)$ (plus an all-to-all sum for domain decomposition methods), and (iii) the factorization of the fully-populated coarse matrix $(\const^T \op \const)$. Then, the application of the projector consists simply of (i) one block dot-product $((\op\const)^Tx)$ (plus an all-to-all exchange), (ii) the resolution of the coarse problem, and (iii) the matrix-vector product $(\const\alpha)$.

Thus, provided that the number of columns of matrix $\const$ is small, the main cost is related to the calculation of $\op \const$. One must bear in mind that block operations (on ``multivectors'') are comparatively much faster than single vector operations, especially when the matrices are sparse (because data fetching is factorized). In a domain decomposition context, product~$\op \const$ corresponds to the resolution of Dirichlet or Neumann problems in substructures, which makes the simultaneous treatment of many columns very efficient (and minimizes the number of exchanges). One must also remember that a conjugate gradient iteration involves a preconditioning step which may be expensive. (The cost is comparable to that of an operator product in optimal domain decomposition methods.) Thus, the additional cost of augmentation relative to the cost of one iteration depends on many parameters (the size of the problem, the number of augmentation vectors, the number of subdomains, the preconditioner chosen...). Typically, in the examples presented in this paper, we found that, using an optimal preconditioner, the CPU cost of between 4 and 7 augmentation vectors (depending on the hardware configuration) cost no more than one CG iteration.

A question which is not addressed in this paper is the verification of the full-rank property of matrix $\const$, which affects the quality of the factorization of matrix $(\const^T \op \const)$. Strategies to correct a dependence among the columns of matrix $\const$ due to inexact arithmetic can be found in \cite{Aliaga00alanczos-type}.

\section{Total reuse of Krylov subspaces}
In this section, we show how it is possible to define efficient augmentation strategies in a multiresolution context. Let us consider the sequence of linear systems:
\begin{equation}
\opk x^{(k)}=b^{(k)} \quad, \quad k=1, \ldots p
\end{equation}
where $A^{(k)}$ is an $n \times n$ symmetric positive definite matrix and $b^{(k)}$ is the right-hand side. Each linear system is solved using an augmented preconditioned conjugate gradient algorithm APCG($\opk$, $\precok, \const^{(k)}$, $b^{(k)}$). Let $m^{(k)}$ be the number of iterations which is necessary to reach convergence, and let
\begin{equation}
\Wdes^{(k)}_m=\left[w_0^{(k)}, \ldots,  w_{m^{(k)}-1}^{(k)} \right]
\end{equation}
be a basis of the associated Krylov subspace.

As explained in the previous section, augmentation never increases the condition number which governs the asymptotic convergence rate. More precisely, the presence of active eigenvectors of the current preconditioned problem in $\const^{(k)}$ may increase the efficiency of the iterative solver significantly. Classical strategies can be used in the case of invariant preconditioned operators ($\op^{(k)}=\op$, $\precoM^{(k)}=\precoM$) and multiple right-hand sides.

It is more difficult to define efficient strategies in the general case of varying operators with no information available on their evolution. A simple and natural idea is to reuse previous Krylov subspaces. A first algorithm which reuses all the previous Krylov subspaces is Total Reuse of Krylov Subspaces (TRKS) (Algorithm \ref{TRKS}), which needs only a few comments:
\begin{itemize}
\item Since (according to \eqref{relation:acg}) ${\const^{(k)}}^{T}\op^{(k)}\Wdes^{(k)}_m=0$, the vectors of the concatenated matrix $\const^{(k+1)}$ are linearly independent. Therefore, APCG($\op^{(k+1)}$, $\precoM^{(k+1)},
\const^{(k+1)}$, $b^{(k+1)}$) does not break down and converges.
\item The previous Krylov subspaces are fully reused through concatenation without post-processing; the only downside is that the memory requirements increase due to the need to save the Krylov subspaces.
\item If the number of columns of matrix $\constk$ becomes too large, the method may become computationally inefficient, even though the number of iterations decreases considerably. Nevertheless, TRKS probably leads to the best reduction in the number of iterations achievable by reusing Krylov subspaces. Therefore, it can be used as a reference in terms of the reduction of the number of iterations for any other algorithm based on a reuse of Krylov subspaces.
\item One possible way to reduce the cost of TRKS without reducing the size of $\const$ consists in using approximate solvers, as in the IRKS strategy \cite{Risler:2000:IAA}.
\end{itemize}

\begin{algorithm2e}[ht]\caption{TRKS-APCG}\label{TRKS}
Initialize $\const^{(0)}=\const_0$ (an $n\times m_0$ full-rank matrix)\;%
\For{$k=0,\ldots,p-1$}{%
Solve $\op^{(k)}x^{(k)} = b^{(k)}$ \\ $\;\;\;$ with APCG($\opk$, $\precok, \constk$, $b^{(k)}$)\;%
Define $\Wdes^{(k)}_m=\left[\hdots, w_j^{(k)} ,\hdots\right]_{0\leqslant j< m^{(k)}}$\;%
Concatenate: $\const^{(k+1)}=\left[\constk,\Wdes^{(k)}_m \right]$}%
\end{algorithm2e}

In order to reduce the cost associated with the total reuse of Krylov subspaces, we propose to work on extracted sub-subspaces, an operation often referred to as the recycling of Krylov subspaces. The objective is to retain the smallest number of independent vectors which achieve the greatest decrease in the number of iterations. Clearly, the most effective approach would be to calculate approximate eigenvectors from the previous Krylov subspaces for the current operator. However, because of the variability of the operators, the extraction of such information would be extremely time consuming and would affect the global efficiency. Conversely, approximate eigenvectors of previous problems can be calculated from the associated Krylov subspaces at nearly no cost. In the following section, we describe an efficient algorithm for the extraction of such approximation vectors along with a simple selection procedure to recycle only a few of these vectors. Of course, the performance of our method depends on the stability of the eigenspaces from one system to another. This topic, especially concerning the lower part of the spectrum, is discussed in \cite{kilmer2006}.

\section{Selective recycling of Krylov subspaces}

The standard convergence of conjugate gradients corresponds to an asymptotic convergence rate. Using this property to predict the number of iterations $n_{_{\epsilon}}$ which is required to reach an accuracy level $\epsilon_{_{cg}}$ leads to a huge overestimation. Indeed, one has:
\begin{equation}\label{eq:res-conv-classique}\begin{aligned}
&\frac{\Vert x_{_{i}}-x \Vert_{_{\A}}}{\Vert x_{_{0}}-x \Vert_{_{\A}}} \leqslant 2  (\sigma_{_{1,n}})^i \leqslant \epsilon_{_{cg}} \quad \Rightarrow \quad
i \geqslant n_{_{\epsilon}} = \frac{ \ln( \epsilon_{_{cg}}/2)}{\ln(\sigma_{_{1,n}})}\\
&\mbox{ with } \sigma_{_{r,s}} = \frac{\sqrt{\kappa_{_{r,s}}} - 1}{\sqrt{\kappa_{_{r,s}}} + 1} \mbox{ and } \kappa_{_{r,s}}= \frac{\lambda_{_{r}}}{\lambda_{_{s}}} \end{aligned}
\end{equation}

This result alone cannot explain the improvement in the convergence rate observed during the iteration process. This superconvergence phenomenon can be explained by a study of the convergence of Ritz values \cite{MORGAN:1995:Eigen} which enables one to define an instantaneous convergence rate \cite{VANDERSLUIS:1986:RCCG}. This explanation can be improved by a study of the influence of the distribution of the eigenvalues \cite{NOTAY:1993:CRCGRE,AXELSSON:1986:RCPCG}.

The objective of recycling Krylov subspaces is to find the best augmentation space in order to trigger superconvergence quickly. This section is organized as follows: we start with a review of Ritz eigenelement analysis and continue with a brief presentation of the improved convergence results; these results lead to a number of selection strategies, which will be assessed in Section \ref{sec:assess}.

\subsection{Ritz analysis: theory and practical calculation}\label{subs:ritz}

For $0\leqslant i<m$, Ritz vectors $(\hat{y}_m^i)$ and values $(\theta_m^i)$ are approximations of the eigenvectors and eigenvalues of the symmetric positive definite matrix $\hat{\op}$; their definition is similar to that of the iterates in the conjugate gradient algorithm \eqref{eq:krysol_principle}
\begin{equation}\label{eq:kryeig_principle}
\left\{\begin{array}{ll}
\text{find }& (\hat{y}_m^i,\theta_m^i) \in \Ks_m(\hat{\op},\hat{v}_0)\times\R\\
\text{such that }& \hat{\op} y_m^i - \theta_m^i y_m^i \perp \Ks_m(\hat{\op},\hat{v}_0)
\end{array}\right.
\end{equation}
The symmetric Lanczos algorithm \cite{SAAD:1992:NMLEP} enables one to build a particular orthonormal basis of $\Ks_m(\hat{\op},\hat{v}_0)$, denoted $\Vritz_m$. Then, the search principle becomes:
\begin{equation}
  y_m^i=\hat{\Vritz}_m q_m^i,\quad  \hat{\Vritz}_m^T\hat{\op}\hat{\Vritz}_m q_m^i = \theta_m^i q_m^i
\end{equation}
The Lanczos basis $\hat{\Vritz}_m$ makes the Hessenberg matrix $\hat{\Hess}_m=\hat{\Vritz}_m^T\hat{\op}\hat{\Vritz}_m$ symmetrical and tridiagonal. $\hat{\Vritz}_m$ and $\hat{\Hess}_m$ can be recovered directly from the conjugate gradient coefficients \cite{SAAD:2000:IMS}:
\begin{equation}\begin{aligned}
\left\{\begin{array}{l}
 \hat{\Vritz}_m = \left(\hdots,(-1)^j\frac{\hat{r}_j}{\|\hat{r}_j\|} ,\hdots\right)_{0\leqslant j< m} \\
 \hat{\Hess}_m=\operatorname{tridiag}(\eta_{j-1},\delta_j,\eta_j)_{0\leqslant
j< m}
\end{array}\right.\\
\text{with }\delta_0 = \frac{1}{\alpha_0},\ \delta_j =
\frac{1}{\alpha_j}+\frac{\beta_{j-1}}{\alpha_{j-1}},\
\eta_j=\frac{\sqrt{\beta_{j}}}{\alpha_j}
\end{aligned}
\end{equation}
Since matrix $\hat{\Hess}_m$ is symmetrical and tridiagonal, its eigenelements $(\theta^m_j,q^m_j)_{1\leqslant j\leqslant m}$ can be calculated easily, for example using a Lapack procedure. Let us define $\bTheta_m=\operatorname{diag}(\theta^1_m \leqslant \ldots\leqslant \theta_m^m)$ and $\bQ_m=\left[q_m^1,\ldots, q_m^m\right]$ such that $\hat{\Hess}_m=\bQ_m \bTheta_m {\bQ_m}^T$. $\bTheta_m$ and $\hat{\Yritz}_m= \hat{\Vritz}_m \bQ_m$ are the Ritz values and associated Ritz vectors, which are approximations of the eigenelements of operator $\hat{\op}$ and satisfy:
\begin{equation*} \hat{\Yritz}_m^T \hat{\op}\hat{\Yritz}_m = \bTheta_m \qquad\text{and}\qquad \hat{\Yritz}_m^T\hat{\Yritz}_m =  \Id_m\end{equation*}

We presented Ritz analysis for the equivalent symmetric system described previously because symmetry simplifies the calculation of eigenelements, but the analysis can be transferred back to the left-preconditioned system using the following transformation rules:
\begin{equation}\label{eq:vritz:2}
\begin{aligned}
    \Vritz_m &= \precoL^{-T}\hat{\Vritz}_m = \left[\ldots, (-1)^{j}\frac{{z}_j}{({r}_j,z_j)^{1/2}}, \ldots\right]\\
    \Hess_m &= \hat{\Hess}_m = \Vritz_m^T \op \Vritz_m\\
    \Yritz_m &= \precoL^{-T}\hat{\Yritz}_m=\Vritz_m\bQ_m
\end{aligned}
\end{equation}
The Ritz vectors are the solution of a generalized eigenproblem and satisfy the following orthogonality properties:
\begin{equation}
  \Yritz_m^T \op \Yritz_m = \bTheta_m \qquad\text{and}\qquad \Yritz_m^T \precoM \Yritz_m = \Id_m
\end{equation}

One can show that when $m$ increases the Ritz values converge toward the eigenvalues of $\hat{\op}$, and that the convergence is either from above or from below depending on their rank \cite{VANDERSLUIS:1986:RCCG,JIA:2004:CRV}:
\begin{equation} \theta^{{1}}_{m} \geqslant \theta^{{1}}_{m-1}  \geqslant \theta^{{2}}_{m} \geqslant \ldots \geqslant \theta^{{m-1}}_{m}  \geqslant \theta^{{m-1}}_
{m-1} \geqslant \theta^{{m}}_{m}
\end{equation}
In addition, in the case of clearly distinct eigenvalues, the convergence of a Ritz value results in the convergence of the associated Ritz vector.

\subsection{Relation between the convergence of conjugate gradients and the convergence of the Ritz values}\label{subs:ba}

In \cite{VANDERSLUIS:1986:RCCG}, the superconvergence phenomenon is explained by the convergence of the Ritz values through the definition, at each iteration, of a instantaneous convergence rate associated with the part of the spectrum that is not yet approximated correctly by the Ritz values: at a given conjugate gradient iteration, one can find a deflated system (with some of its extreme eigenvalues removed) with similar behavior. Let $\left[\lambda_{_{l}} , .., \lambda_{_{r}}  \right]$ be the spectrum of the deflated operator. The equivalent convergence rate is:
\begin{equation}\label{eq:cvcgritz} \Vert x-x_{_{i+1}} \Vert_{_{\A}}  \leqslant \; F_{_{i,l,r}} \; 2 \; \sigma_{_{l,r}}
\Vert x-x_{_{i}} \Vert_{_{\A}}  \end{equation}
where $F_{_{i,l,r}}$ quantifies the convergence of the $l$ smallest and $r$ largest Ritz values to the extreme eigenvalues:
\begin{equation*} \begin{aligned} F_{_{i,l,r}} &=  \max_{l'>l } J^{(i)}_{_{l,l'}} \;  \max_{r' \geqslant r} L^{(i)}_{_{r,r'}}\\J^{(i)}_{_{l,l'}} &= \prod_{_{j=1}}^{l}
\left| 1 - \frac{\lambda_{_{l'}}}{ \lambda_{_{j}}}\right|  \;
\left| 1 - \frac{\lambda_{_{l'}}}{\theta_{j}^{i}}\right|^{-1} \\
L^{(i)}_{_{r,r'}}  &= \prod_{_{j=1}}^{r}
 \left|   1 -  \frac{\lambda_{_{n-r^{'}}}}{\lambda_{_{n+1-j}}} \right| \;
 \left|  1 -  \frac{\lambda_{_{n-r^{'}}}}{\theta_{i+1-j}^{i}} \right|^{-1}\end{aligned}\end{equation*}
Since this result holds for every pair $(l,r)$, the effective convergence rate at Iteration $i$ corresponds to the pair $(l,r)$ which minimizes $\sigma_{_{i,l,r}} =  F_{_{i,l,r}} \sigma_{_{l,r}}$.

Then, after some iterations, the superconvergent conjugate gradient algorithm behaves very much like a conjugate gradient algorithm augmented by the extreme eigenvectors which are associated with the converged Ritz values.
In a multiresolution context, provided the linear systems have similar spectral properties, the Ritz vectors associated with the converged Ritz values obtained for one system should define a viable augmentation space for the subsequent resolutions.

\subsection{Effect of the distribution of the eigenvalues}\label{subs:evdba}

The effect of the distribution of the eigenvalues on the convergence of conjugate gradients was studied in \cite{NOTAY:1993:CRCGRE,AXELSSON:1986:RCPCG}. The results take into account the fact that preconditioning often leads to clustered eigenvalues as opposed to uniformly distributed eigenvalues, as can be seen in Figure~\ref{fig:eigen-distrib}.

In addition to other results, the authors showed that if a spectrum consists of $p$ isolated eigenvalues in the high part of the spectrum, $p$ isolated eigenvalues in the low part of the spectrum and $n-2p$ uniformly distributed central eigenvalues, then the conjugate gradient convergence takes the form:
\begin{equation}
n_{_{\epsilon}} \geqslant \tilde{n}_{_{\epsilon}}  =  2p +\\
 \mathrm{int} \; \left( \frac{ \ln \; (\epsilon_{_{cg}} /2 )}{ \ln \; \sigma_{_{p+1,n-p}} }
- \frac{  \sum_{i=1}^{p} \; \ln  \left(\frac{\lambda_{_{n-p+i}}}{4\lambda_{_{i}}}\left(1 - \frac{\lambda_{_{i}}}{\lambda_{_{n-p+i}}}  \right) \right)}{\ln \; \sigma_{_{p+1,n-p}} } \right)
\label{eq:conv-cg-petitesgrandesvp}
\end{equation}
The convergence rate is approximately equal to the classical convergence rate for the central part, plus one iteration per higher eigenvalue and a little more than one iteration per lower eigenvalue. These results can be combined with the work by Jiao \cite{JIA:2004:CRV,JIA:2004:CHRV} on the convergence of Ritz values. In general, since the method is related to the power iteration method, a correct approximation by the Ritz values is obtained first for the highest eigenvalues, then for the lowest part of the spectrum, resulting in superconvergence (which is governed by the asymptotic convergence rate of the reduced spectrum).

\subsection{Selection procedures}\label{subs:sel-proc}

The results of Section~\ref{subs:ba} lead to a first proposal of a selection procedure for converged Ritz vectors: convergence is identified by the stagnation of the Ritz values; if the conjugate gradient algorithm converges at iteration $m$, the Ritz values are calculated for the previous two states $\bTheta_m$ and $\bTheta_{m-1}$.
 Once ranked, the $m$ most recent Ritz values $\bTheta_m$ are compared to the $m-1$ previous values according to the following criteria:
\begin{equation}
\left\{\begin{aligned}
&\theta_{m}^j \text{ has converged if }  \frac{\vert \theta_{m}^j-\theta_{m-1}^j \vert}{\vert \theta_{m}^j \vert} \leqslant \varepsilon,\qquad 1 \leqslant j\leqslant m\!-\!1 \\
&\theta_{m}^{m-j} \text{ has converged if }  \frac{\vert \theta_{m}^{m-j}-\theta_{m-1}^{m-1-j} \vert}{\vert \theta_{m}^{m-j} \vert} \leqslant \varepsilon,\qquad 0 \leqslant j\leqslant m-2
\end{aligned}\right.
\end{equation}
where $\varepsilon$ is a user parameter which is easy to adjust since the criterion is generally either very high (before the convergence of the Ritz value) or very small (after convergence). Figure \ref{fig:eigen-distrib} illustrates that property with the simple example of the operator associated with the decomposition of a linear elastic cube into ten subdomains; in that case, the higher half of the spectrum has converged.
\begin{figure}[ht]
\begin{center}
\includegraphics[width=0.6\linewidth]{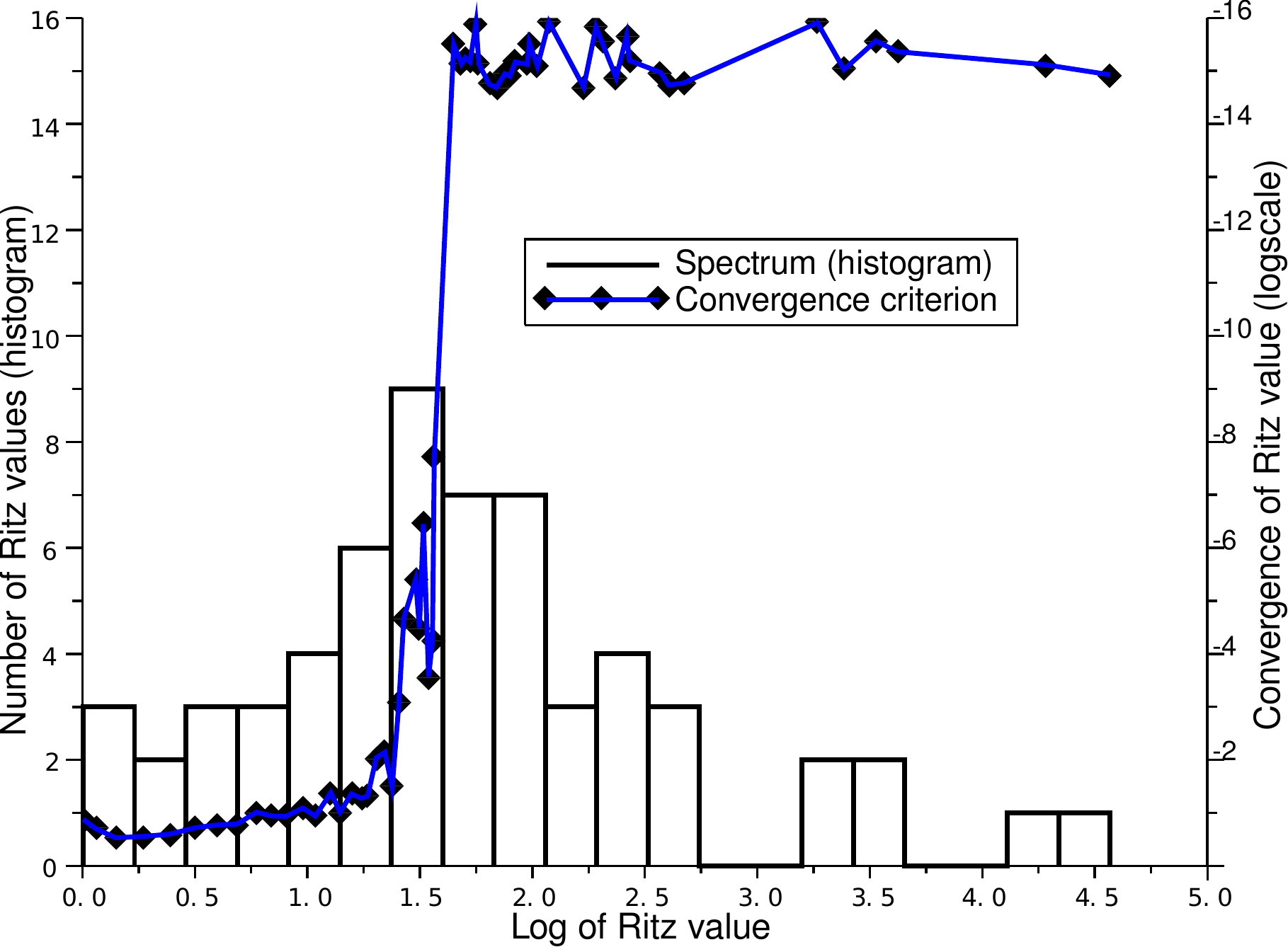}
\end{center}
\caption{Ritz spectrum and convergence of the Ritz values}
\label{fig:eigen-distrib}
\end{figure}

The principle of the selective recycling of Krylov subspaces (SRKS-APCG) is described in Algorithm~\ref{SRKS}. Basically, in addition to the memory required by APCG, the SRKS-APCG algorithm requires storage for $m$ $n$-vectors $(z_j)_{j=1,m}$. One should note that the selected vectors are normalized by the square root of the associated Ritz value in order to improve the condition number of the coarse matrix. (If operator $\op$ remained constant, matrix $(\const^T \op\const)$ would be the identity matrix.)

\begin{algorithm2e}[h!t]\caption{SRKS-APCG}\label{SRKS}
Initialize $C^{(0)}=\const_0$ (full column rank matrix)\;
\For{$k=0,\ldots,p-1$}{%
$\bullet$ Solve $\op^{(k)}x^{(k)} = b^{(k)}$ with APCG($\opk, \precok, \constk, b^{(k)}$)\;

$\bullet$ Define $ {\Vritz}_m=\left[\hdots, (-1)^{j}\frac{{z}_j}{({r}_j,z_j)^{1/2}} ,\hdots\right]_{0\leqslant j< m}$\;

$\bullet$ Define $ {\Hess}_m=\operatorname{tridiag}(\eta_{j-1},\delta_j,\eta_j)_{0\leqslant j< m}$ \, 
$ \delta_0 = \frac{1}{\alpha_0} \; ,  \; \delta_j =  \frac{1}{\alpha_j} + \frac{\beta_{j-1}}{\alpha_{j-1}} \; , \;  \eta_j=\frac{\sqrt{\beta_{j}}}{\alpha_j}$\;

$\bullet$ Compute eigenelements $ (\bQ_m,\bTheta_m)$ of ${\Hess}_m$ $(\theta_m^1\geqslant \ldots \geqslant \theta^m_m)$\;

$\bullet$ Compute $\Yritz_m = \Vritz_m\bQ_m=\left[y_m^1, \ldots , y_m^m \right]$\;

$\bullet$ Extract ${\Hess}_{m-1}=\operatorname{tridiag}(\eta_{j-1},\delta_j,\eta_j)_{0\leqslant j< m-1}$ \;

$\bullet$ Compute eigenvalues $(\theta_{m-1}^j)$ of ${\Hess}_{m-1}$ \;

\For{$j=1,\ldots,m-1$}{%
$\const=\left[\const, \frac{y_m^j}{\sqrt{|\theta_{m}^j|}} \right]$ if $|\theta_{m}^j-\theta_{m-1}^j| \leqslant \varepsilon |\theta_{m}^j| $\;
$\const=\left[\const, \frac{y_m^{j+1}}{\sqrt{|\theta_{m}^{j+1}|}} \right]$ if $|\theta_{m}^{j+1}-\theta_{m-1}^{j}| \leqslant \varepsilon |\theta_{m}^{j+1}|$
} %
$\bullet$ Concatenate $ {\const}^{(k+1)}=[\const^{(k)};{\const}]\;\; , \;\; \const=\left[0\right] $\;

$\bullet$ If dim(${\const}^{(k+1)}) \geqslant n_{c_{lim}}$, then ${\const}^{(k)}=\const^{(0)}$
}
\end{algorithm2e}

For better computational efficiency, a restart parameter can be introduced in order to limit the size of the augmentation space associated with parameter $n_{c_{lim}}$ in Algorithm~\ref{SRKS}. This limit size can be set after a complexity analysis under the assumption that all non-augmented systems would be solved in the same number of iterations. However, we did not use such a restart procedure in our experiments.

In order to be even more selective, we propose a reselection strategy based on a prediction of the efficiency of the retained vectors. Indeed, the results of Section~\ref{subs:evdba} in terms of the effect of the distribution of the eigenvalues lead us to retain only the converged Ritz vectors which belong to the external part of the spectrum:
\begin{itemize}
\item this is known to be the first part of the spectrum whose approximation by Ritz values is good;
\item since the convergence of Ritz vectors is identified by the stagnation of the associated Ritz values, the fact that the external Ritz values are distinct ensures that the Ritz vectors approximate the eigenvectors correctly \cite{JIA:2004:CRV};
\item while choosing vectors in the dense central zone does not modify the shape of the spectrum and does not improve convergence, selecting the external part of the spectrum triggers superconvergence instantly.
\end{itemize}

In order to select only the external part of the spectrum, we implemented the cluster identification algorithm proposed in \cite{MOLINARI:2001:MTCD}. This algorithm seeks the piecewise constant distribution which is nearest (in a least squares sense) to the distribution of the distances among the sorted eigenvalues. The only parameter required is the minimum size of the cluster, which we set at one-fifth the number of preselected vectors. As will be shown in the next section, the performance achieved with this reselection algorithm is not outstanding, but some results in terms of gain per augmentation vector are worth considering.

\section{Numerical assessments}\label{sec:assess}

We present three numerical experiments. Two concern the evaluation of a structure made of random materials, as is the case in a Monte-Carlo simulation. In the first case, the materials are elastic; in the second case, which is a nonlinear problem, they are elastic-plastic. The last case is a large displacement problem, which raises specific difficulties.

The methods were implemented in the ZEBULON code \cite{ZEBUUSER:2001} and parallelism was introduced using MPI. The calculations were performed on the LMT-Cachan cluster, which consists of dual quadcore and dual hexacore processors connected by a gigabit network. The calculations were always carried out on homogeneous sets of processors which were entirely dedicated to one task which fit entirely in memory, so swapping was not necessary. In each case, we indicate the CPU time which measures the amount of work performed for one subdomain. The Wall Clock Time (WCT), a global measure which is more sensitive to external perturbations induced by the operating system and the presence of other users, was considered to be unreliable in many cases; so we mention it only for the first set of experiments. One should note that the gains calculated with WCT were always greater.

The CPU plots show the total time as well as the time dedicated to augmentation (preparation of the coarse operator, initialization and projections); the difference represents the iterations of the solver.

All the calculations used a dual formulation of the interface problem through domain decomposition (FETI). The convergence was evaluated using the norm of the residual (which corresponds to the displacement gaps at the interfaces) normalized by the condensed right-hand side. Classically for such structural problems, total reorthogonalization was used to enforce the $\op$-conjugation of the search directions. (The case without reorthogonalization is discussed briefly in the first example.)

\subsection{The case of a sequence of linear systems}
We considered a cube (of side $50$ mm) with $4\times4\times4=64$ small cubic inclusions (of side $5.5$ mm). A slice through this structure is shown in Figure~\ref{fig:cube}. The cube was clamped over one side, and the opposite side plus another side were subjected to uniform pressure. The mesh consisted of $125,000$ linear hexahedral elements for a total of $400,000$ degrees of freedom. Three automatic decompositions (into 12, 48 and 96 subdomains) were performed using the Metis algorithm \cite{METIS:1998} (see Figure~\ref{fig:cubedec}). The resulting interface system contained $54,000$ unknowns for the 12-subdomain decomposition, $96,000$ unknowns for the 48-subdomain decomposition and $133,000$ unknowns for the 96-subdomain decomposition. All the materials were isotropic, linear and elastic, and were characterized by their Young's modulus and Poisson's coefficient. The material properties of each inclusion and of the matrix were chosen randomly following a normal law with a relative standard deviation equal to $10\%$, leading to a $\pm 23\%$ variation range about the nominal value. The average Young's modulus was $200$ MPa for the matrix and $20,000$ MPa for each inclusion, and the average Poisson's coefficient was $0.27$ for the matrix and $0.35$ for each inclusion. The objective was to perform the calculations for $40$ draws of the $130$ coefficients.
 \begin{figure}[ht]\centering
 \begin{minipage}{.49\textwidth}\centering
    \includegraphics[width=0.9\textwidth]{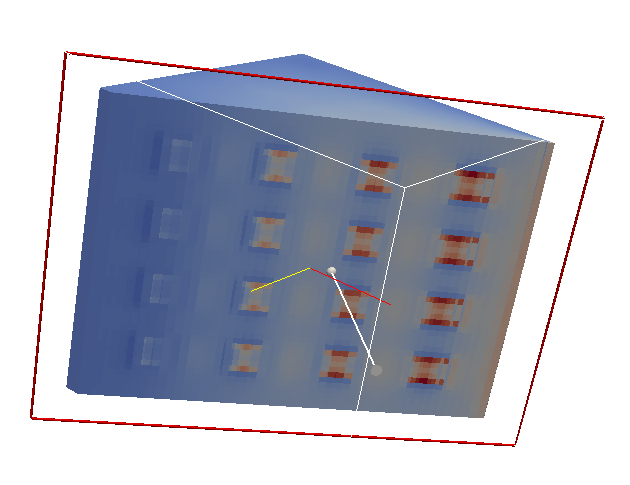}
    \caption{A slice through the heterogeneous cube (shear stress)}  \label{fig:cube}
    \end{minipage}\hfill
     \begin{minipage}{.49\textwidth}\centering
    \includegraphics[width=0.8\textwidth]{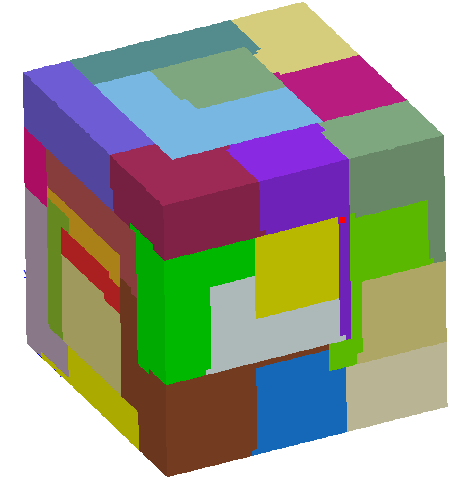}
    \caption{The decomposition into 48 subdomains}  \label{fig:cubedec}
    \end{minipage}
\end{figure}

We used a dual formulation (FETI) with both a Dirichlet (optimal) and a lumped preconditioner, leading to $10^{-3}$ and $10^{-6}$ APCG accuracy respectively.
We considered the following algorithms: conjugate gradients (cg), total reuse of Krylov subspaces (trks) and selective reuse of Krylov subspaces with two values of the criterion, $\varepsilon=10^{-6}$ (srks6) and $\varepsilon=10^{-14}$ (srks14). In addition, in the last case ($\varepsilon=10^{-14}$), we also attempted to further refine the selection by not selecting the converged Ritz values contained in the central cluster (identified by the algorithm proposed by \cite{MOLINARI:2001:MTCD}); this method is labeled (clust14).
\subsubsection{Comparison of the strategies}\

The results for the 12-subdomain decomposition are summarized in Table~\ref{tab:cubinc}, which gives the average number of APCG iterations to convergence, the average size of the augmentation space, the final size of the augmentation space, the average CPU and wall clock times per system, from which we also deduced the augmentation time (operator preparation and projection). Computations were conducted one dual hexacore processor (one subdomain per core). When the average and final sizes of the augmentation space are close, this means that most of the augmentation space was identified with the first systems. For a given configuration (accuracy and preconditioner), the figures in bold in the three columns `average number of iterations', `average CPU time' and `average wall clock time' indicate the best strategy in terms of gain per unit augmentation vector compared to CG.

For the 12-subdomain decomposition, Figures~(\ref{fig:cub3i}, \ref{fig:cub3a}, \ref{fig:cub3t}) (for an objective of $10^{-3}$ accuracy) and Figures~(\ref{fig:cub6i}, \ref{fig:cub6a}, \ref{fig:cub6t})) (for an objective of $10^{-6}$ accuracy) give the evolutions of the number of APCG iterations to convergence for each linear system, the dimension of the augmentation space $n_c$ and the the CPU time for the resolution of each system, with both lumped and Dirichlet preconditioners. \medskip

\begin{table}\centering
\begin{tabular}{|c|c|c||p{.7cm}|p{.7cm}|p{.7cm}|p{.7cm}|p{.7cm}|p{.7cm}|p{.7cm}|}\hline
\multicolumn{3}{|c||}{} & \multicolumn{7}{|c|}{12 subdomains} \\\hline
accur. &  precond. &   & avg. \# it & avg. $n_c$ & max $n_c$ &  avg. total CPU & avg. CPU aug. & avg. total  WCT& avg. WCT aug. \\\hline
\multirow{12}*{$10^{-3}$}	&	\multirow{5}*{\rotatebox{90}{Dirichlet}}	
   &	cg (no reo.)	&	70	& ---& ---& 23.5 & 0 & 31.8 & 0\\
    &    &	cg		&	44.3	& ---& ---& 16.0 &0& 24.5 &0 \\
	&	&	trks	&	2.4	& 77.8 & 96 &	7.1 & 6.3 & 9.1 & 7.6\\
	&	&	srks6	&	20.5 & 41.3 & 50 &	11.0 & 3.8 & 15.6 & 4.7\\
	&	&	srks14	&   25.6 & 24.6 & 27 & \textbf{11.7} & 2.7 &\textbf{17.1} & 3.5\\
	&	&	clust14	&   \textbf{30.6}	& 16.8 & 24 & 13.8 & 2.2 & 25.3 & 3.2 \\\cline{2-10}
	&	\multirow{5}*{\rotatebox{90}{Lumped}}	
       &	cg (no reo.)	&	145	& ---& ---& 27.7 & 0 & 40.3&0 \\
    & & cg		&	68.1	& ---& ---& 13.6 & 0 & 24.0&0\\
	&	&	trks	&	\textbf{0.4} & 71.8 & 74 &	\textbf{5.7 } &5.6& \textbf{6.8}&6.7\\
	&	&	srks6	&	27.4	& 81 & 108&	12.3 &6.5& 18.0&8.0 \\
	&	&	srks14	&   32.0	& 59.6 & 71 & 11.8 & 5.1 & 18.1 &6.4\\
	&	&	clust14	&   51.1	& 20 & 39 & 12.8 & 2.3& 21.7&3.2 \\\hline
\multirow{12}*{$10^{-6}$}	&	\multirow{6}*{\rotatebox{90}{Dirichlet}}	
        &	cg (no reo.)	&	174	& ---& ---& 58.3 &0& 85.8&0 \\
     &   &	cg		&	84.4	& ---& ---& 30.0 &0& 49.1& 0\\
	&	&	trks	&	13.2	& 382.9& 551&	46.5 &41.3& 60.4&52.8\\
	&	&	srks6	&	36.7	& 104.3& 142&	22.3 &8.6& 35.7&11.0 \\
	&	&	srks14	&   42.8	& 72.6 & 87 & 22.7 &6.3& 36.2& 8.1\\
	&	&	clust14	&   \textbf{60.8}	& 35.2 & 77 & \textbf{25.7} & 3.6&\textbf{39.8 }& 4.9\\\cline{2-10}
	&	\multirow{5}*{\rotatebox{90}{Lumped}}	
        &	cg (no reo.)	&	{\tiny $>$}400& ---& ---& {\tiny $>$}78&0&{\tiny $>$}110 & 0 \\
    &   &	cg		&	147.7	& ---& ---& 31.7 &0& 61.0&0 \\
	&	&	trks	&	16.3	& 516.7& 735&70.2 &65.7& 93.3&85.9\\
	&	&	srks6	&	54.2	& 225.7& 311&	33.4 &20.5& 49.4&25.6\\
	&	&	srks14	&   60.6	& 170.2& 216& 29.2 &15.0 &44.2&18.3\\
	&	&	clust14	&   \textbf{129.4}	& 20 & 39 & \textbf{30.4} &2.6& \textbf{55.6}&4.0 \\\hline
\end{tabular}
\caption{Performance summary for the cube with inclusions}\label{tab:cubinc}
\end{table}
\begin{figure}[ht]
\begin{minipage}{.5\textwidth}\centering
\includegraphics[width=0.99\textwidth]{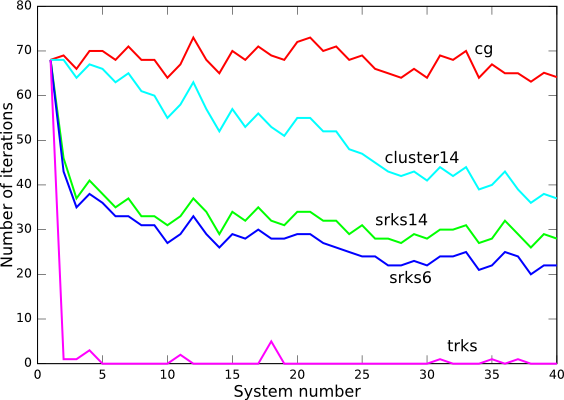}
    \caption{Cube 12 subdomains, lumped, $10^{-3} accuracy$, number of iterations per linear system}  \label{fig:cub3i}
\end{minipage}
\begin{minipage}{.5\textwidth}\centering
\includegraphics[width=0.99\textwidth]{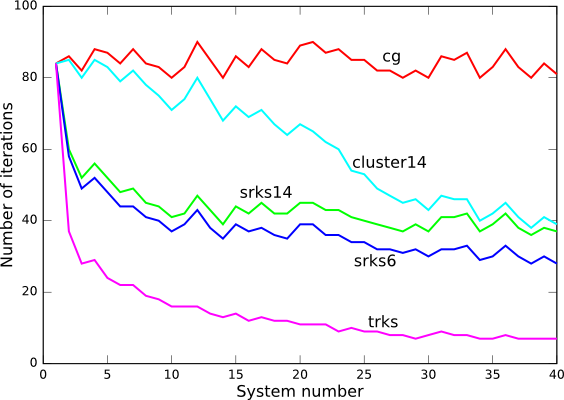}
    \caption{Cube 12 subdomains, Dirichlet, $10^{-6} accuracy$, number of iterations per linear system}  \label{fig:cub6i}
\end{minipage}
\end{figure}
\begin{figure}[ht]
\begin{minipage}{.5\textwidth}\centering
\includegraphics[width=0.99\textwidth]{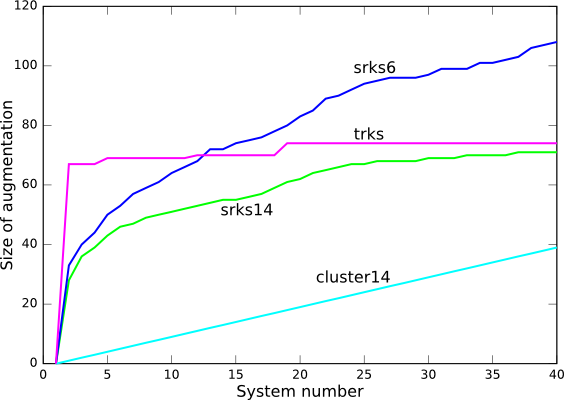}
    \caption{Cube 12 subdomains, lumped, $10^{-3} accuracy$, dimension of the augmentation space}  \label{fig:cub3a}
    \end{minipage}
\begin{minipage}{.5\textwidth}\centering
\includegraphics[width=0.99\textwidth]{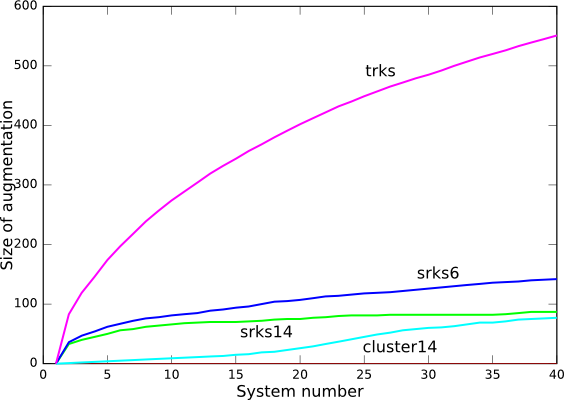}
    \caption{Cube 12 subdomains, Dirichlet, $10^{-6} accuracy$, dimension of the augmentation space}  \label{fig:cub6a}
    \end{minipage}
\end{figure}
\begin{figure}[ht]
\begin{minipage}{.5\textwidth}\centering
\includegraphics[width=0.99\textwidth]{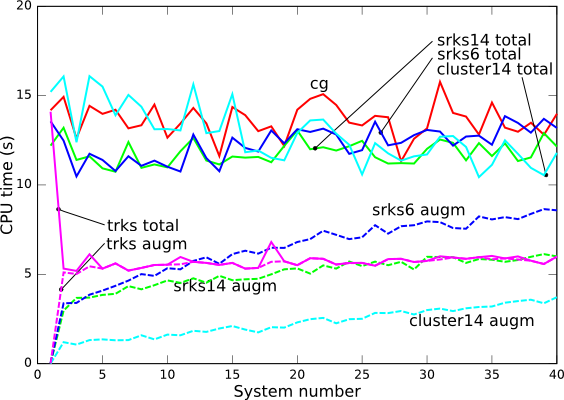}
    \caption{Cube 12 subdomains, lumped, $10^{-3} accuracy$, CPU time per linear system}  \label{fig:cub3t}
    \end{minipage}
\begin{minipage}{.5\textwidth}\centering
\includegraphics[width=0.99\textwidth]{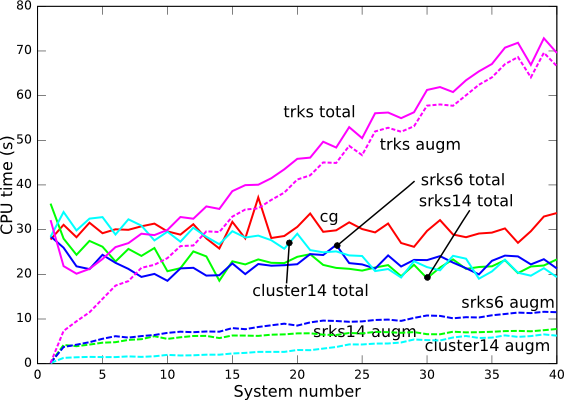}
    \caption{Cube 12 subdomains, Dirichlet, $10^{-6} accuracy$, CPU time per linear system}  \label{fig:cub6t}
    \end{minipage}
\end{figure}

Without full reorthogonalization, the performance was very poor and led to about twice the number of iterations of the recommended fully reorthogonalized conjugate gradients. This was expected because systems resulting from domain decomposition formulations are known to often require full reorthogonalization \cite{FARHAT:1994:ADV}. Furthermore, one should note that the additional iterations carried out in the non-reorthogonalized case led to vector sets which made the Ritz analysis more complex due to the appearance of nonphysical, multiple eigenvalues. The non-reorthogonalized approach was no longer considered in the other examples.

With the TRKS approach, two types of behavior were observed. In the low-accuracy case ($10^{-3}$), for both preconditioners (but especially for the lumped preconditioner), the size of the augmentation space reached a plateau, which means that the augmentation space contained almost all the required information; the gains in terms of both the number of iterations ($>90\%$) and the CPU time ($>55\%$) were excellent. In the high-accuracy case ($10^{-6}$), the size of the augmentation space never stabilized; therefore, even though the number of iterations decreased drastically, the CPU time increased. Table \ref{tab:trksit} gives extended performance results for TRKS which confirm this analysis. The gains are given relative to conjugate gradients. The efficiency of augmentation is defined by the average decrease in the number of iterations per augmentation vector; the higher the required accuracy, the less efficient the TRKS approach. These results justify our decision to select the subspaces so that the dimension of the augmentation space would remain under control.

The SRKS14 approach succeeded in limiting the size of the augmentation space and led to a satisfactory decrease in the number of iterations. As can be seen on the figures, SRKS6 did not stabilize the augmentation space as efficiently and behaved half way between TRKS and SRKS14; therefore, we will choose SRKS14 as our reference algorithm from now on.

The cluster strategy as it stands today gave unsatisfactory results: even though it often led to the best gain per augmentation vector, it seemed to impair the selection of useful vectors and allow much less reduction in the number of iterations than SRKS. After the resolution of many systems, it tended to lead to the same augmentation space as SRKS.

To confirm that hypothesis, we compared the spaces $C_{SRKS}$ and $C_{cluster}$ after the 40 resolutions for the low-accuracy Dirichlet case. We used the following procedure: 
 first, the vectors were orthonormalized using SVD: $C=U\Sigma V^T$; then SVD was applied to the concatenated matrix $[U_{SRKS},U_{cluster}]$. A plot of the singular values is shown in Figure~\ref{fig:svdclustsrks}. Independent spaces would lead to a constant value equal to 1, while for nested spaces the common space would lead to $\{\sqrt{2},0\}$ pairs of singular values. One can observe that the spaces are not exactly nested, but come quite close.

In conclusion, the cluster strategy is not mature yet, but it is promising. It was not considered for the following experiments because, due to the larger number of systems involved, it would behave quite similarly to SRKS.

\begin{figure}[ht]\centering
     \begin{minipage}{.49\textwidth}
    \includegraphics[width=0.9\textwidth]{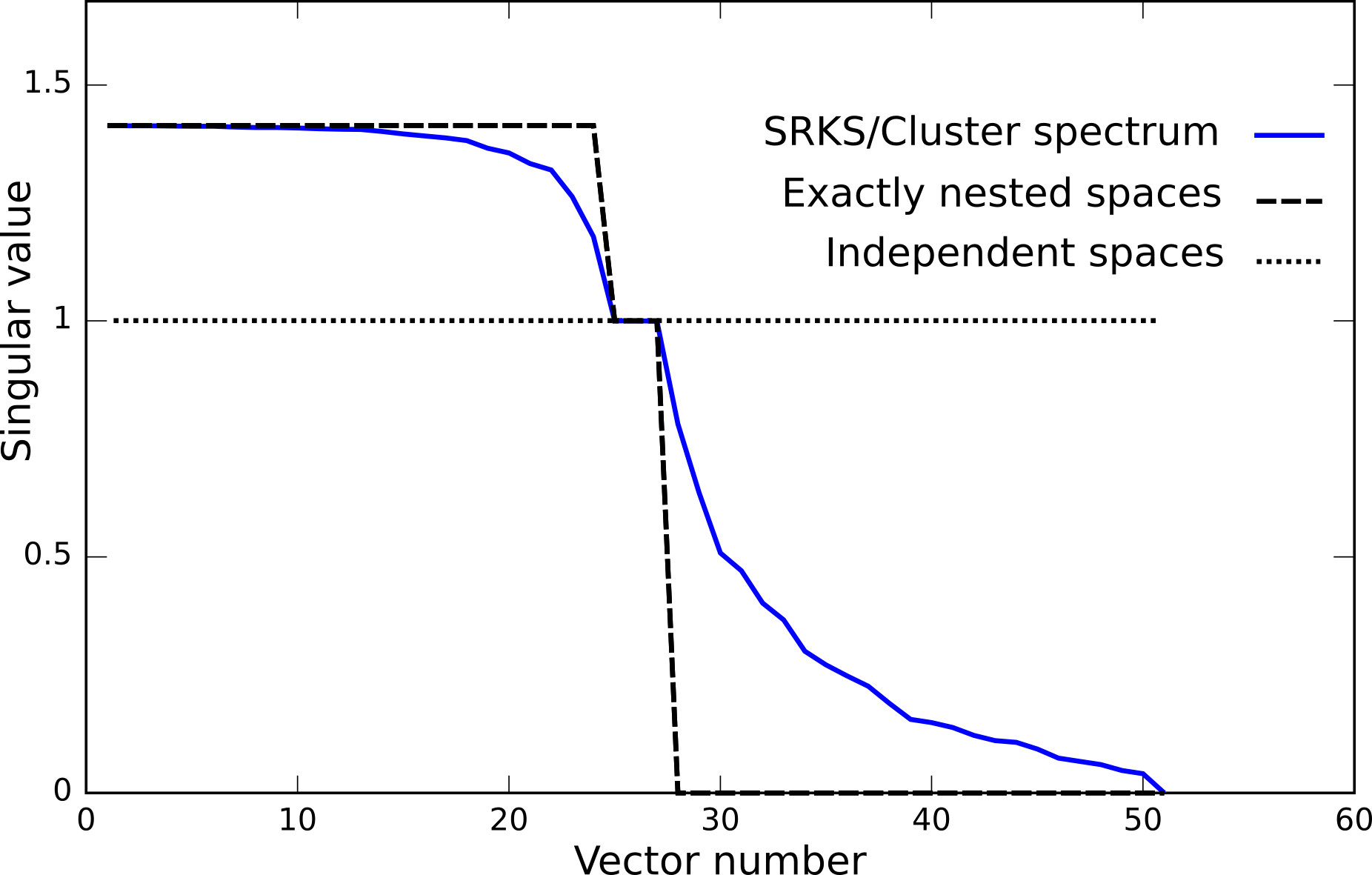}
    \caption{Singular values of $[C_{SRKS},C_{cluster}]$}  \label{fig:svdclustsrks}
    \end{minipage}
\end{figure}

\begin{table}
\centering
\begin{tabular}{|c|c|c|c|c|p{2cm}|}\hline
accur. & precond.	 &  \# subdomains & iteration gain & CPU gain & efficiency of augmentation\\\hline
\multirow{4}*{$10^{-3}$}	&\multirow{2}*{lumped}	&12 &99.4\%	&58.2\%	&0.94 \\
						&							&48	&97.7\%	&60.1\%	&0.97 \\	\cline{2-6}
						&\multirow{2}*{Dirichlet}	&12 &94.7\%	&55.5\%	&0.54 \\
						&							&48	&96.2\%	&62.6\%	&0.69 \\	\hline
\multirow{4}*{$10^{-6}$}&\multirow{2}*{lumped}		&12 &89\%	&-121.2\%	&0.25 \\
						&							&48	&90.2\%	&-162.3\%	&0.37 \\	\cline{2-6}
						&\multirow{2}*{Dirichlet}	&12 &84.4\%	&-55.1\%	&0.19 \\
						&							&48	&83.3\%	&-90.7\%	&0.21 \\	\hline
\end{tabular}\caption{Relative performance of TRKS}\label{tab:trksit}
\end{table}

\subsubsection{Study of SRKS14 in various configurations}\

Table~\ref{tab:srks14it} shows the relative performance of SRKS14 as a function of the number of subdomains, of the preconditioners and of the accuracy. The efficiency of the augmentation is defined as the decrease in the average number of iterations per augmentation vector. One can observe that the efficiency ranged between $0.5$ and $0.85$ and was best for the lower accuracy and the improved preconditioner. For these spectra in which there exist no small isolated eigenvalues (which could lead to efficiencies greater than 1), such results are consistent with the theory (see Section~\ref{subs:evdba}). In the next section, we will see that this moderate efficiency does not preclude significant CPU improvements.

The gains in terms of the number of iterations were relatively stable, typically between $50\%$ and $60\%$ in the high-accuracy case.

\begin{table}[ht]
\centering
\begin{tabular}{|c|c|c|p{1.7cm}|c|c|p{2cm}|}\hline
 precond.	 & accur. & \# subdomains & CG avg. \#~iterations & avg. $n_c$& iteration gain  & efficiency of augmentation\\\hline
\multirow{4}*{lumped}	&\multirow{2}*{$10^{-3}$}	&12&68.1 	&59.6	&52.9\%		&0.6\\
						&							&48& 43.7	&28.2	&54.2\%		&0.84 \\
\cline{2-7}
						&\multirow{2}*{$10^{-6}$}	&12& 147.7	&170.2	&59\%	&0.51\\
						&							&48& 162.7	&186.3	&62.5\%		&0.55 \\
\hline
\multirow{6}*{Dirichlet}&\multirow{3}*{$10^{-3}$}	&12& 43.3	&24.6	&42.2\%		& 0.76\\
						&							&48&50.7	&31	&48.6\%		&0.79 \\
						&							&96&67.4	&51.1	& 51.7\%	& 0.68 \\\cline{2-7}
						&\multirow{3}*{$10^{-6}$}	&12& 84.4	&72.6	&49.3\%		&0.57\\
						&							&48&116.	&111.5	&57.2\%		&0.6 \\
						&							&96&140.7	&141.9	& 60.3\%	& 0.6\\\hline
\end{tabular}\caption{Iteration gains for SRKS14}\label{tab:srks14it}
\end{table}

\subsubsection{Influence of the hardware configuration on the CPU gains}\

Now, let us study the performance of SRKS14 in terms of CPU time for the same decomposition into 48 subdomains, but using different hardware configurations:
\begin{enumerate}
\item Configuration A corresponds to 4 dual hexacore nodes with 1 subdomain per core;
\item Configuration B corresponds to 6 dual quadcore nodes with 1 subdomain per core;
\item Configuration C corresponds to 3 dual quadcore nodes with 2 subdomains per core;
\item Configuration D corresponds to 2 dual quadcore nodes with 3 subdomains per core.
\end{enumerate}
One can note that the processors in Configuration A were different from those used in the other cases.
%
In all the cases, the memory was sufficient to avoid swapping. The results are given in Table~\ref{tab:srks14cpudiri} for the Dirichlet preconditioner and in Table~\ref{tab:srks14cpulumped} for the lumped preconditioner. One can see that Configurations B,C and D had similar performances and were slower than Configuration A due to the different memory technology.

One interesting factor is the ratio of the average CPU cost of an iteration to the average CPU cost of an augmentation vector (the last columns of Table~\ref{tab:srks14cpudiri} and~\ref{tab:srks14cpulumped}). One can see that in Configuration A, 4 augmentation vectors cost no more than one iteration; in the other configurations 7 augmentation vectors cost no more than one iteration. Since we saw that one needs about $1/0.6\simeq 1.6$ augmentation vectors to save one iteration, the advantage of augmentation is clear. Indeed, we observe a 32\% CPU improvement in Configuration A and a $40\%$ to $50\%$ improvement in the other configurations.

Note that when the lumped preconditioner is used the equivalent cost of an iteration is only 2.8 augmentation vectors in Configuration A and 4.5 augmentation vectors in Configuration D (see Table~\ref{tab:srks14cpulumped}). Since the efficiency of the augmentation vectors is less when this inexpensive preconditioner is used (in the high accuracy case), so is the CPU improvement.

\begin{table}[ht]
\centering
\begin{tabular}{|c|c|c|p{3.5cm}|}\hline
 Configuration	 & CG avg.  CPU & CPU gain & CPU~per~iteration / CPU~per~augm.~vector \\\hline
A & 9.5 & 32.6\%& 4 \\\hline
B & 25.7& 41.5\%& 6.7\\\hline
C & 29.7& 48.9\%& 7.7\\\hline
D & 29.9& 47.6\%& 7.8\\\hline
\end{tabular}\caption{CPU performance of SRKS14 for $10^{-6}$ accuracy with the Dirichlet preconditioner}\label{tab:srks14cpudiri}
\end{table}
\begin{table}[ht]
\centering
\begin{tabular}{|c|c|c|p{3.5cm}|}\hline
 Configuration	 & CG avg.  CPU & CPU gain & CPU~per~iteration / CPU~per~augm.~vector \\\hline
A & 10.7& 22.4\%& 2.8 \\\hline
D & 27.4 & 33.3\%& 4.5\\\hline
\end{tabular}\caption{CPU performance of SRKS14 for $10^{-6}$ accuracy with the lumped preconditioner}\label{tab:srks14cpulumped}
\end{table}

The ratio of the CPU time per iteration to the CPU time per augmentation vector for SRKS (Column 4 of the previous tables) turned out to be relatively stable for a given machine with a given preconditioner. This is due to the stability of the size of the augmentation space which prevented the cost from soaring (as would happen with TRKS). Thus, the CPU performance can be deduced from the iteration gains and the augmentation efficiency (see Table~\ref{tab:srks14it}). For instance, the CPU gain for SRKS with the 96-subdomain decomposition was slightly greater than 50\%.

\subsection{The case of a sequence of nonlinear problems}
Now let us consider a hexahedral holed plate ($10\times 10\times 0.2$ mm with a center hole of radius $1$ mm, see Figure~\ref{fig:plaque}) subjected to unidirectional tension (a prescribed normal displacement). The plate was discretized into $61,000$ linear hexahedral elements for a total of $41,000$ degrees of freedom. The structure was divided into 8 subdomains using the Metis algorithm, which resulted in an interface system with $3,000$ unknowns. The problem was solved using one 8-core processor (one subdomain per core). Elastic-plastic behavior with nonlinear isotropic hardening and a Von Mises'-type plasticity criterion was assumed. Denoting $\sigma$ the Cauchy stress tensor, $\epsilon(u)$ the symmetric gradient of the displacement field $u$, and $\mK$ the Hooke tensor, the material law can be written as:
\begin{equation}
\left\{\begin{array}{l}
  \epsilon(u)={\epsilon}^{{e}}+{\epsilon}^{{p}}, \qquad {\sigma} = \mK : {\epsilon}^{{e}} \\
 \text{ if } f(\sigma) =0      \; \text{ then } \; \dot{\epsilon}^{{p}} =   \multlag   f_{_{,\sigma}} \\
 \text{ if } f(\sigma) \leqslant 0 \; \text{ then } \; \dot{\epsilon}^{{p}} = 0 \\
f(\sigma ) =  \sqrt{ \frac{3}{2} \; \sigma :\sigma } - \left(R_{_{0}} + Q \left( 1 -e^{- b \multlag}\right)\right)
\end{array}\right.
\end{equation}
The coefficients were assigned a normal law with a $10\%$ relative standard deviation, which implied variations of up to $\pm 23\%$ in the coefficients. The mean values of the material parameters were: $E=200,000$ MPa,  $\nu=0.3$, $R_0=300$ MPa, $b=22$ and $Q=170$ MPa. The loading was applied in two steps: first, a single increment to reach the elastic limit; then, $16$ equal increments in order to multiply the prescribed displacement by $4$. The objective of the study was to analyze $21$ configurations.

Again, the linear solver used was FETI with a Dirichlet or lumped preconditioner. The accuracy objective for the linear systems was set at $10^{-6}$. (The accuracy must be high for the nonlinear process to run well). Because of the approximations, not all the methods converged in the same number of Newton iterations; on average, one nonlinear analysis required the resolution of 95 tangent systems. Table~\ref{tab:plaque} summarizes the performances of the various methods; Figure~\ref{fig:plate_i} shows the evolution of the average number of APCG iterations with the lumped preconditioner during the sequence of linear systems; Figure~\ref{fig:plate_a} shows the evolution of the size of the augmentation space; Figures~\ref{fig:plate_cpu} and~\ref{fig:plate_wall} show the evolutions of the average CPU time and wall clock time for the resolution of one linear system along with the evolution of the average augmentation time (operator creation and projection).

\begin{minipage}{.3\textwidth}\centering
    \includegraphics[width=0.99\textwidth]{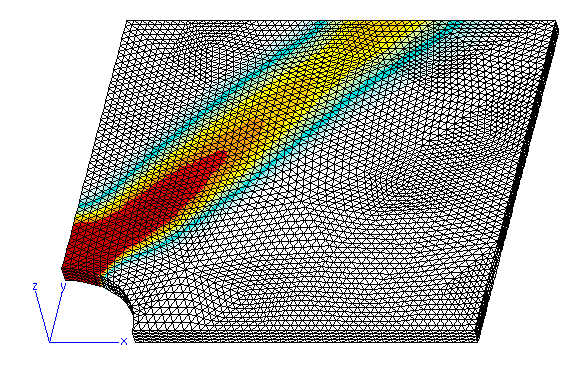}
    \captionof{figure}{The holed plate example (plastic strain)}  \label{fig:plaque}
    \end{minipage}
\begin{minipage}{.7\textwidth}\centering
\begin{tabular}{|c|c|p{.65cm}|p{.65cm}|p{.65cm}|p{.65cm}|p{.65cm}|}\hline
precond. & method  & avg. \# it & avg. $n_c$ & max $n_c$ &   avg. CPU & avg. WCT\\\hline
Dirichlet & cg & 25.6 & -- & -- & 1.21 & 3.03 \\
Dirichlet & trks$^*$ & 1.4 & 358 & 492 &  2.66 &  9.83 \\
Dirichlet & srks14 & \textbf{16.1} & 17 & 19 &  \textbf{0.98} & \textbf{2.35}\\\hline
lumped & cg & 41.4 & -- & -- &  1.03 & 3.24\\
lumped & trks$^*$ & 1.2 & 520 & 695 &  4.72 &6.98\\
lumped & srks14 & \textbf{19.1} & 43 & 45 &  \textbf{0.87}&\textbf{2.08}\\\hline
\end{tabular}

{\footnotesize $^*$ calculation too slow, was stopped before all the systems were solved
}
\captionof{table}{Holed plate, performance summary}  \label{tab:plaque}
\end{minipage}

\begin{figure}[ht]
\begin{minipage}{.5\textwidth}\centering
\includegraphics[width=0.99\textwidth]{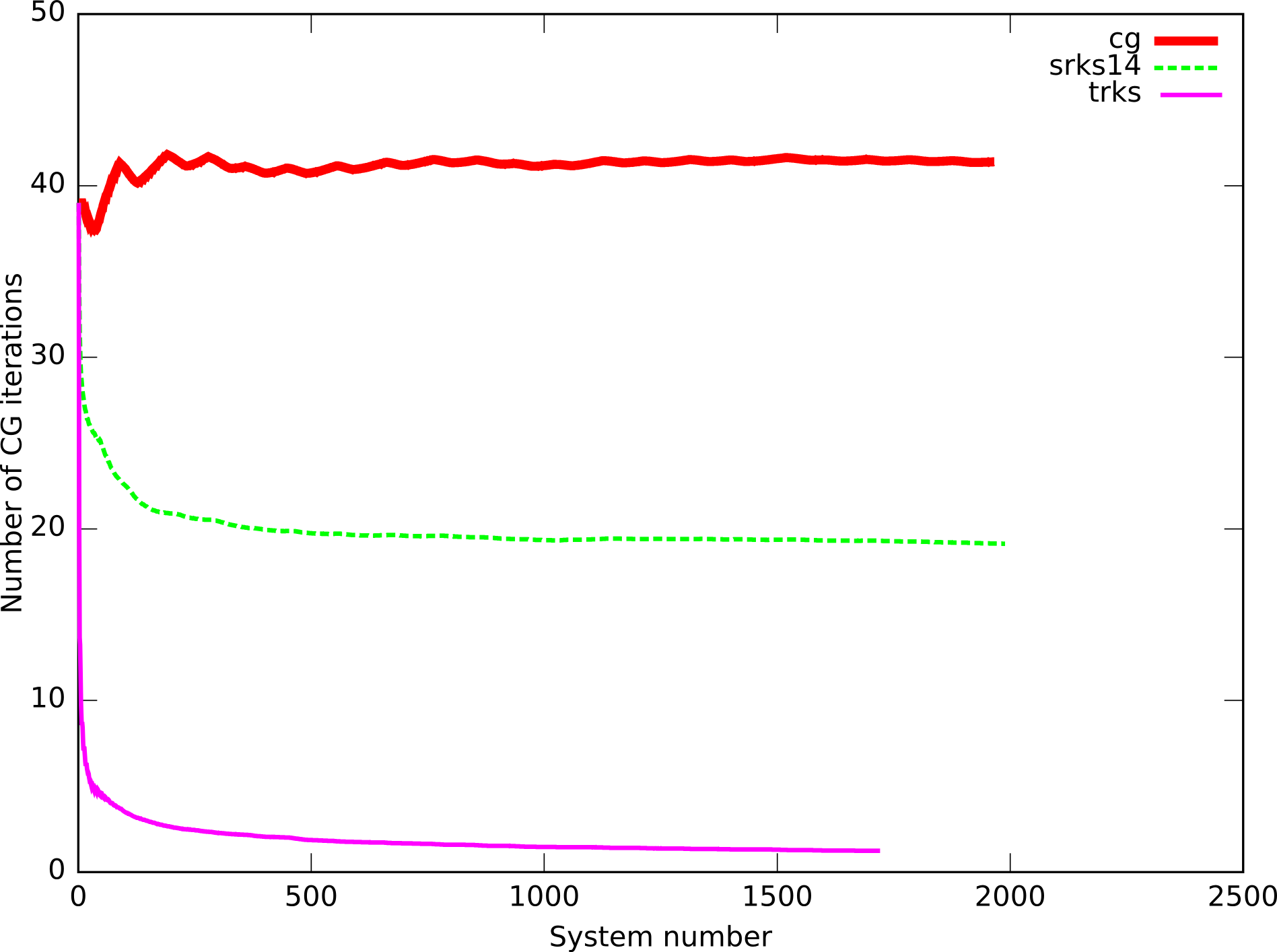}
    \caption{Plate, lumped -- avg. \# it. / linear system}  \label{fig:plate_i}
\end{minipage}
\begin{minipage}{.5\textwidth}\centering
  \includegraphics[width=0.99\textwidth]{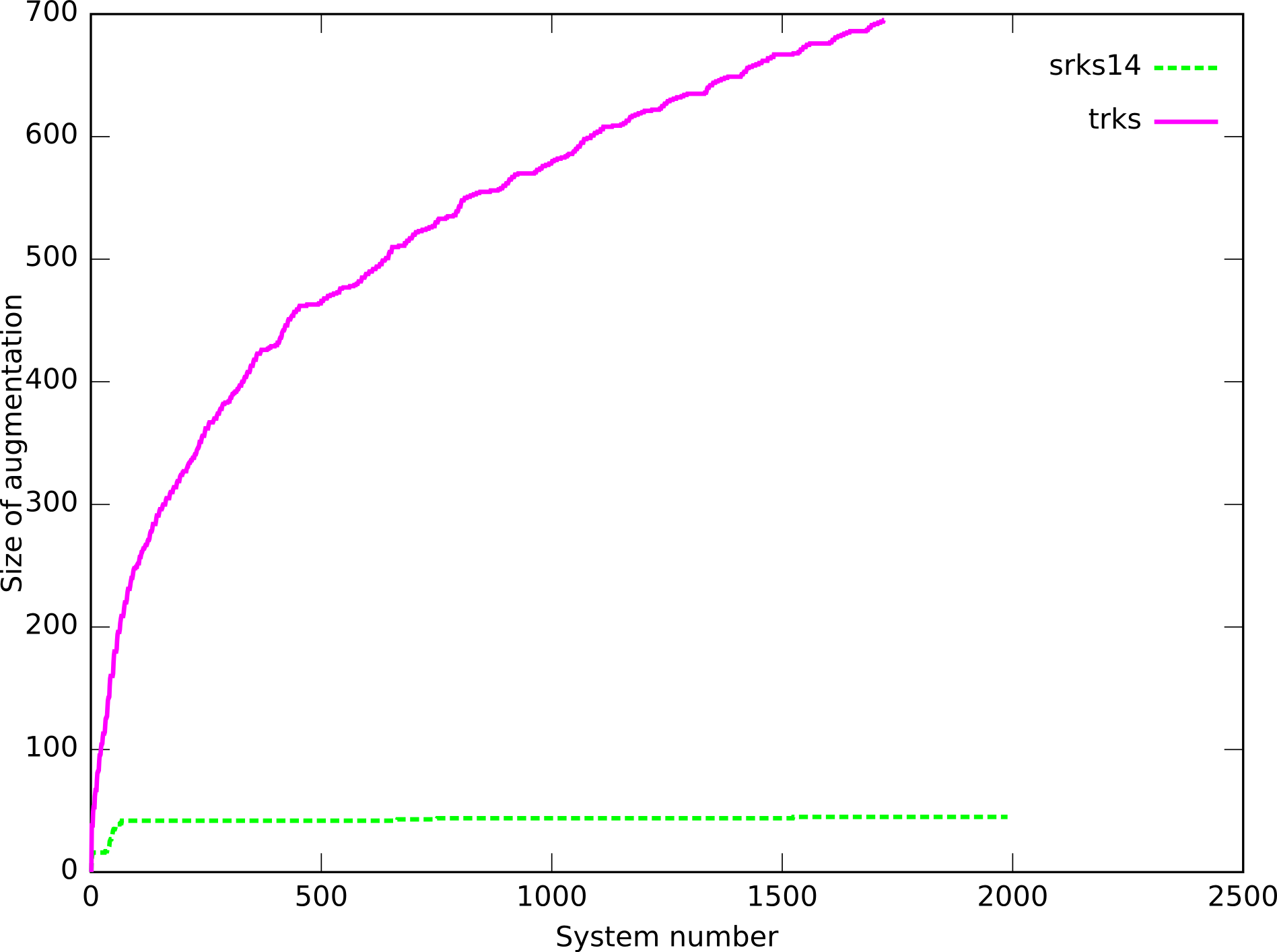}
    \caption{Plate, lumped -- dimension of aug. space}  \label{fig:plate_a}
\end{minipage}
\end{figure}
\begin{figure}[ht]
\begin{minipage}{.5\textwidth}\centering
\includegraphics[width=0.99\textwidth]{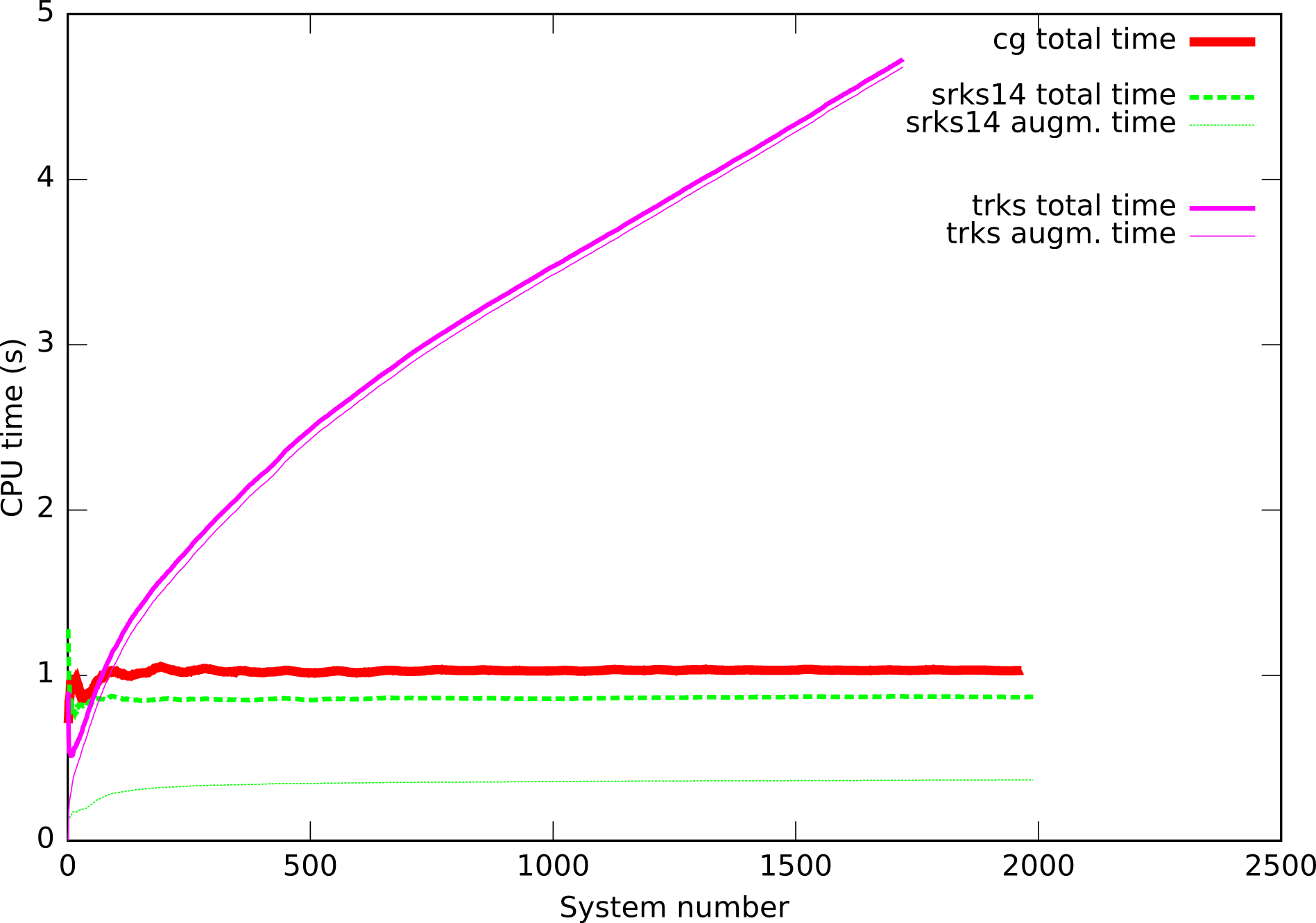}
    \caption{Plate, lumped -- avg. CPU time / system}  \label{fig:plate_cpu}
\end{minipage}
\begin{minipage}{.5\textwidth}\centering
  \includegraphics[width=0.99\textwidth]{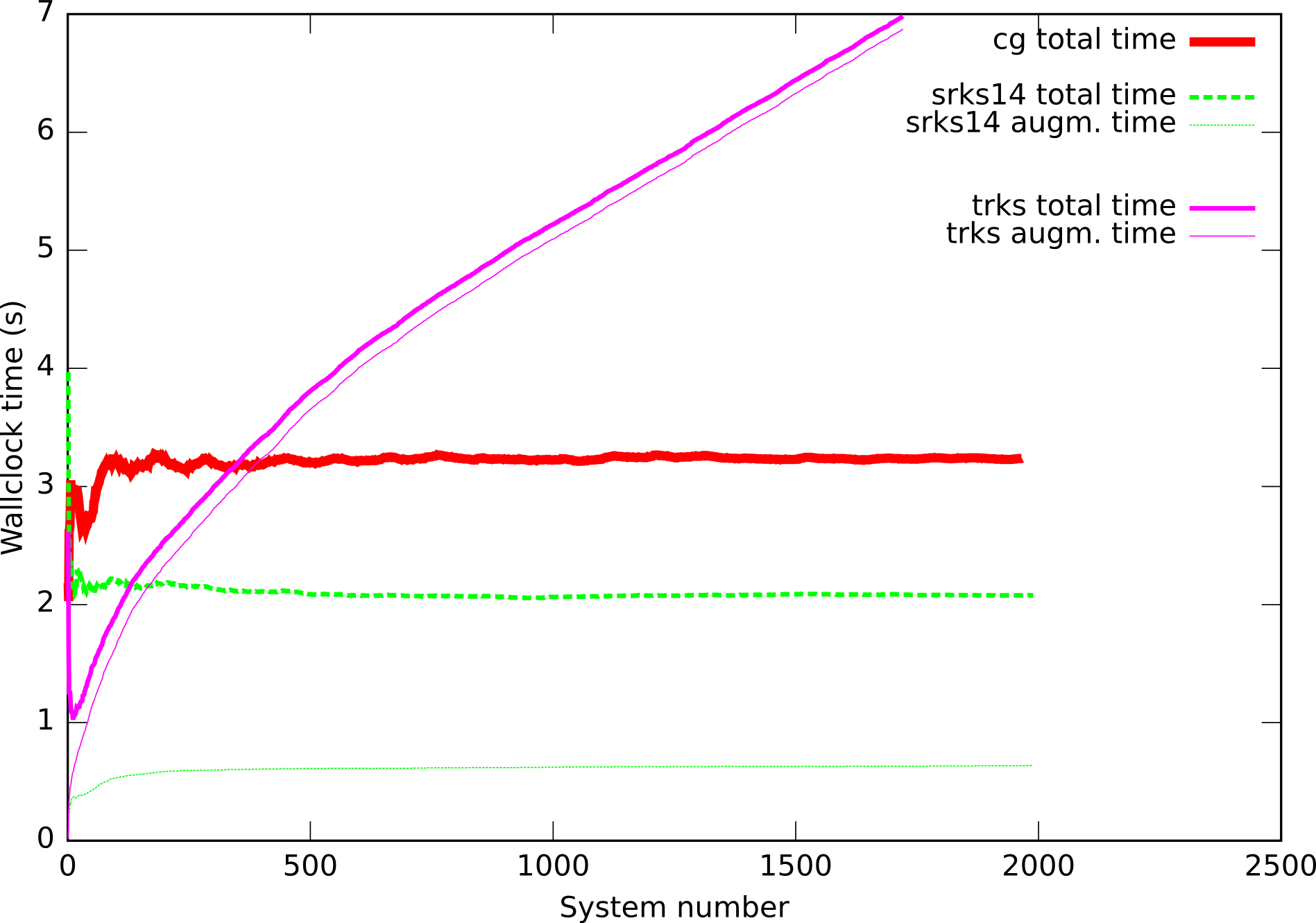}
    \caption{Plate, lumped -- avg. wall clock time / system}  \label{fig:plate_wall}
\end{minipage}
\end{figure}

This test leads to conclusions similar to the previous ones. More specifically, one can observe that the SRKS augmentation space selected after the first nonlinear configuration remained stable. Conversely, since the TRKS augmentation space never reached a plateau, the solutions of the linear systems did not belong to a common space. SRKS was the most efficient method, leading to a 20\% CPU gain and a 36\% wall clock time improvement.

\subsection{The case of a large displacement problem}

Finally, let us consider the problem of the buckling of a straight heterogeneous beam with a circular cross section (length/diameter ratio equal to $30$), clamped at one end and subjected to an axial pressure at the other, with no radial displacement. The heterogeneities consisted of five straight fibers whose stiffness was $1,000$ times that of the matrix. The problem was formulated in the updated Lagrangian framework, assuming linear elastic behavior (characterized by the Young's modulus and Poisson's coefficient) in the current configuration. The beam was discretized into $90,000$ linear hexahedral finite elements for a total of $300,000$ degrees of freedom. It was divided into $10$ subdomains using the Metis algorithm, leading to an interface system with $16,000$ unknowns. A single 12-core processor was used (1 subdomain per core, leaving 2 inactive cores). The pressure was applied incrementally up to the configuration shown in Figure~\ref{fig:pout_flamb}, in which the maximum axial displacement was about $3\%$ of the total length. 12 increments were used, leading to the resolution of about 30 tangent linear systems.

We used a FETI solver with a Dirichlet preconditioner and an ``identity'' projector. The FETI convergence criterion was set to $10^{-6}$. Figure~\ref{fig:pout_flamb_iter} shows the evolution of the number of conjugate gradient iterations required for the resolution of each linear system. Figure~\ref{fig:pout_flamb_aug} shows the evolution of the size of the augmentation space.
\begin{figure}[ht]
\begin{minipage}{.5\textwidth}\centering
   \includegraphics[width=0.99\textwidth]{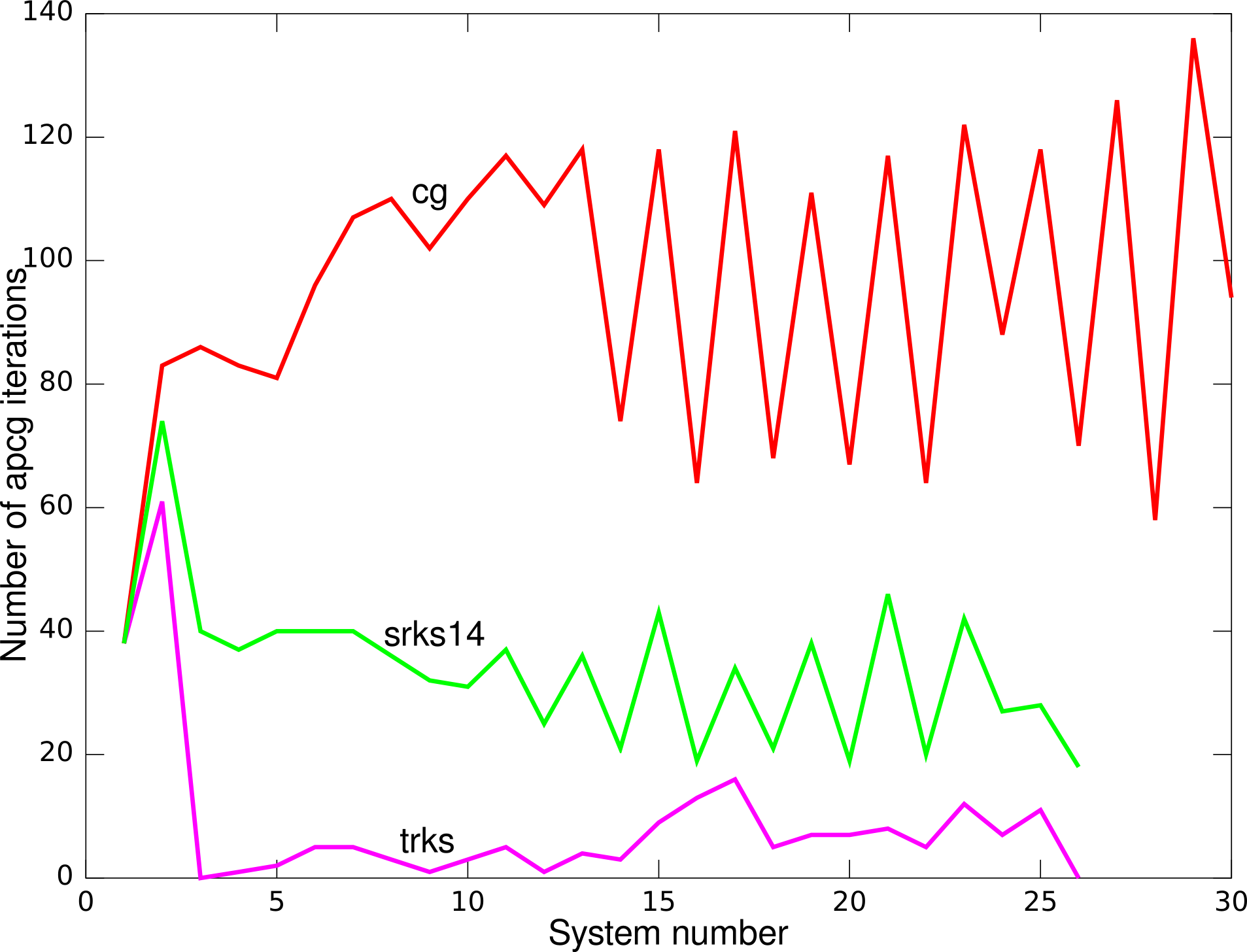}
   \caption{Buckling of the heterogeneous beam: number of iterations per linear system}  \label{fig:pout_flamb_iter}
\end{minipage}
\begin{minipage}{.5\textwidth}\centering
   \includegraphics[width=0.99\textwidth]{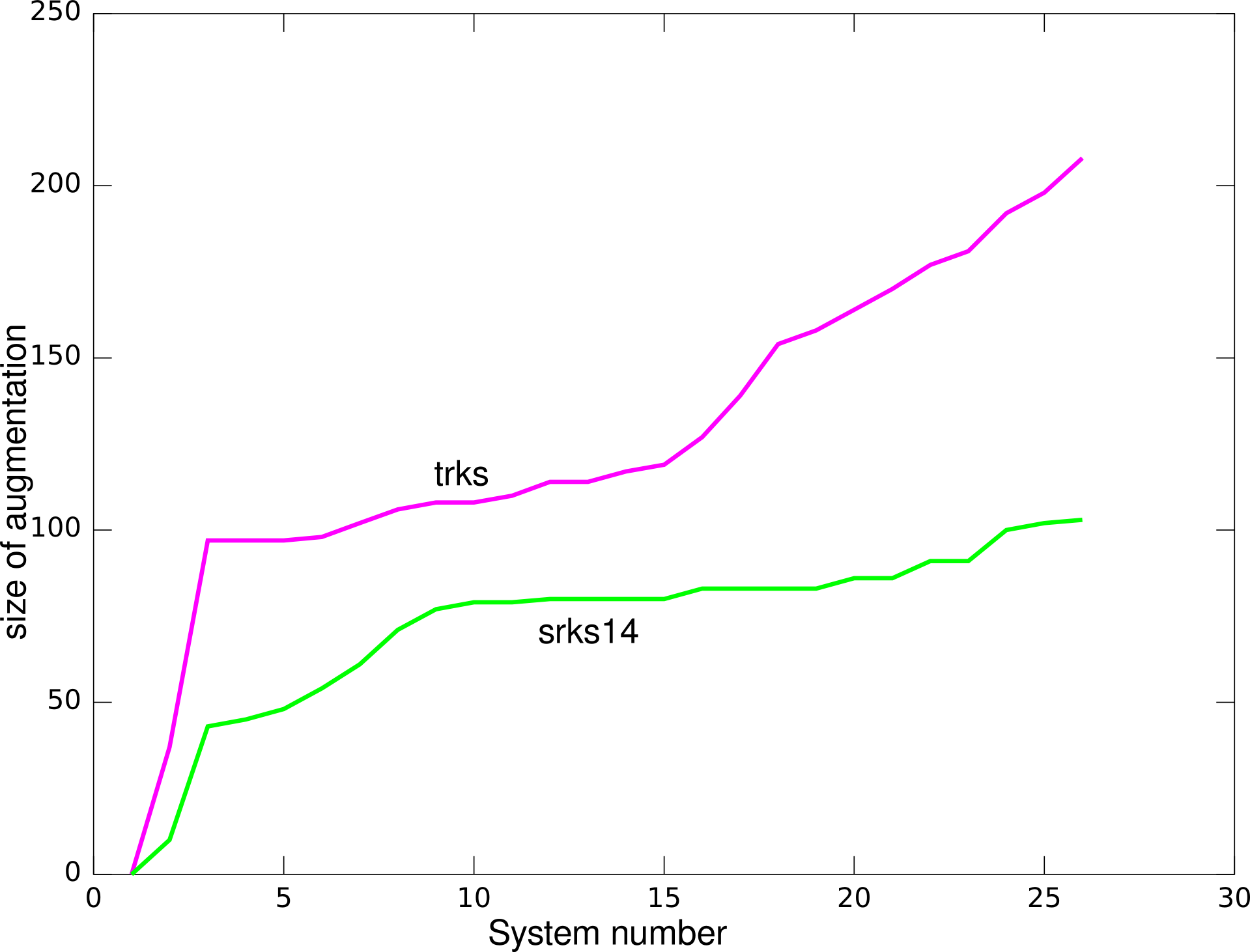}
   \caption{Buckling of the heterogeneous beam: dimension of the augmentation space for each linear system}  \label{fig:pout_flamb_aug}
\end{minipage}
 \end{figure}
Three algorithms were tested: classical conjugate gradients, total reuse of subspaces, and selective reuse of subspaces ($\varepsilon=10^{-14}$). Table~\ref{tab:pout} summarizes the main results.

\begin{figure}[ht]\centering
\begin{minipage}[b]{.45\textwidth}\centering
   \includegraphics[width=0.95\textwidth]{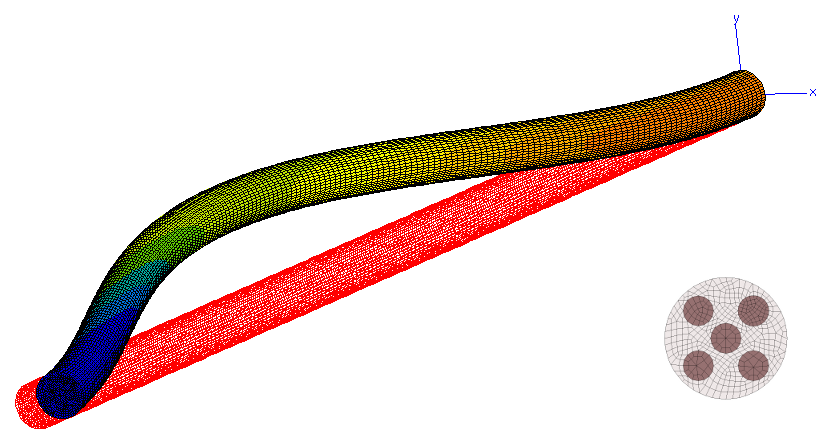}
   \captionof{figure}{The beam in the reference and deformed configurations, with a view of the cross section}  \label{fig:pout_flamb}
\end{minipage}\hfill
\end{figure}
\begin{table}[ht]\centering
\begin{minipage}[b]{.5\textwidth}\centering
  \begin{tabular}{|c|c|c|c|c|}\hline
   method      & avg. \# it. & avg. $n_c$  & CPU & WCT\\\hline
  CG & 97.2 & -- & 392 & 504 \\
TRKS & 7.8 & 131.7 & 203 & 217 \\
SRKS 14 & \textbf{33.8} & 75.1 & \textbf{278} & \textbf{298}\\\hline
  \end{tabular}
  \captionof{table}{Recycling performance for the buckling problem}\label{tab:pout}
  \end{minipage}
\end{table}

The following observations can be made:
\begin{itemize}
  \item The performance of TRKS was really impressive because the systems converged in 10 times fewer iterations than with CG, but the size of the associated augmentation space was large (up to 208 vectors) and never ceased to increase;
  \item SRKS appeared to be efficient: the number of iterations was divided by 3 with a space whose size increased slowly, then reached a plateau;
  \item With the hardware configuration used, the best results in terms of computation time (a 48\% CPU improvement) were achieved with TRKS, but the gain normalized by the number of augmentation vectors was better with SRKS (a 29\% CPU improvement). SRKS was truly successful in controlling the dimension of the augmentation space.
\end{itemize}

In this example, the augmentation proved to be very efficient in terms of the iteration gain per augmentation vector, especially for SRKS (0.85). This was probably because of a specificity of the spectrum of the preconditioned operator due to the use of domain decomposition in large displacements. Indeed, the tangent matrix of a floating subdomain with prescribed Neumann conditions may become non-positive, contrary to the Dirichlet operator which remains positive definite. It is known that a slight lack of positivity of the operator does not prevent reorthogonalized conjugate gradients from converging \cite{PAIGE:1995:ASE}, but convergence is slower than when all the eigenvalues are positive. The positivity of the preconditioner makes the selection procedure still possible. Moreover, the negative eigenvalues are systematically selected by the procedure, so using augmentation causes the solver to iterate in the subspace in which the operator is positive, leading to a much better convergence rate.

\section{Conclusion}
This paper dealt with the resolution of sequences of large linear systems with varying matrices and right-hand sides using conjugate gradients. We proposed several algorithms based on the augmentation of the current Krylov subspace by a selection of previously generated subspaces. The advantage of these methods is that some of the iterations are replaced by the preprocessing of a coarse problem associated with optimized operations.

When low accuracy is sufficient, total reuse of the previous subspaces (the TRKS algorithm) appears to lead to satisfactory results. When high accuracy is required, the subspaces are too unstable, which causes the dimension of the TRKS augmentation space to soar. Therefore, we proposed to retain only the part of the subspace generated by the Ritz vectors associated with converged Ritz values of the preconditioned operator (the SRKS algorithm). These vectors can be built very inexpensively. Such an augmentation was found to remain stable throughout the linear systems and to lead to a reduction in the number of iterations which is consistent with the theory. In terms of computation time, the proposed method leads to a variable, but always positive, gain compared to non-augmented systems. We observed CPU time improvements of 20\% to 50\%, and wall clock time improvements of 40\% to 70\%.

Up until now, our attempts to improve the selection algorithm by eliminating the converged values in the central part of the spectrum have not led to impressive results. This probably means that the Ritz vectors associated with converged Ritz values contain meaningful information which cannot be removed from the analysis of the current system. A continuation of this work could consist in a better analysis of the accumulation of the augmentation vectors. This can be done by studying the coarse matrix $\const^T\op\const$ whose distance to the identity matrix characterizes the variation of the Krylov subspaces. Another objective would be to port some of the ideas presented in this paper to nonsymmetric solvers.

\bibliographystyle{plain}
\bibliography{krylov}

\end{document}